\documentclass{amsart}
\usepackage[utf8]{inputenc}
\usepackage{hyperref}
\usepackage{amsmath, amsthm, amssymb}
\usepackage{tikz}
\usepackage{tikz-cd}
\usepackage{tikz-3dplot}
\usepackage{euscript}
\usepackage{biblatex}

\theoremstyle{plain}
\newtheorem{theorem}{Theorem}[section]
\newtheorem{corollary}[theorem]{Corollary}
\newtheorem{proposition}[theorem]{Proposition}
\newtheorem{lemma}[theorem]{Lemma}
\newtheorem*{thm*}{Theorem}

\theoremstyle{definition}

\newtheorem{remark}[theorem]{Remark}
\newtheorem{definition}[theorem]{Definition}

\newtheorem{example}[theorem]{Example}
\hypersetup{%
	colorlinks,%
	linkcolor={red!60!black},%
	citecolor={red!60!black},%
	urlcolor={red!60!black}%
}

\def\on{\operatorname}
\def\scr{\EuScript}
\def\bb{\mathbb}
\def\sf{\mathsf}
\def\cal{\mathcal}

\def\RR{\bb{R}}

\def\ZZ{\bb{Z}}
\def\NN{\bb{N}}
\def\Sim{\on{Sim}}

\def\TC{\sf{TC}}
\def\LS{\sf{LS}}
\def\CCC{\sf{CC}}

\def\secat{\sf{secat}}
\def\sd{\sf{sd}}
\def\SC{\sf{SC}}
\def\star{\sf{Star}}

\title[Finite Equivariant Topological Complexity]{A Finite Equivariant Generalization of Motion Planning and Topological Complexity}

\author{Rebecca Bell, Allison Eckert, Ryan Pesak, Avery Schweitzer}
\date{}

\begin{document}

\maketitle

\begin{abstract}
This paper explores topological complexity in the finite equivariant setting. We first define and study an equivariant version of Tanaka's combinatorial complexity for finite topological spaces. We explore the relationships between this invariant and several others already discussed in the literature: Farber's topological complexity, Tanaka's combinatorial complexity, and Colman-Grant's equivariant Lusternik-Schnirelmann category. We find bounds for equivariant combinatorial complexity and for the necessary lengths of equivariant combinatorial motion plannings. We show that the equivariant topological complexity of any finite $G$-space is equal to its equivariant combinatorial complexity.

We then adapt Gonz\`{a}lez's simplicial complexity to ordered and unordered $G$-simplicial complexes and explore its first properties. Lastly, we show that the equivariant topological complexity of the realization of any ordered $G$-simplicial complex is equal to the equivariant simplicial complexity. 
\end{abstract}

\section{Introduction}\label{sec:Intro}
The topological complexity of a space $X$, denoted $\TC{(X)}$, is a homotopy invariant that was first introduced by Farber in 2003 in \cite{farber_2003}. Loosely speaking, the topological complexity measures how navigationally complex the space is by computing how many different algorithms one needs to assign a path from $x$ to $y$ to any pair of points $(x,y)\in X\times X$. Formally, it is defined to be the minimal number $k$ such that $X\times X$ can be covered by $k$ open subsets, each of which needs only one continuous rule to assign a path.

Since Farber's introduction of $\TC(X)$, a number of variants on topological complexity have appeared in the literature. Of particular interest for the purposes of this paper is the work of Colman and Grant in \cite{colman_grant_2012}, which develops a notion of \textit{equivariant topological complexity} for a space $X$ equipped with the continuous action of a topological group G. Denoted $\TC_G(X)$, this invariant could be said to measure how many rules are needed to plan paths in X in a way that respects the symmetries of $X$ encoded in the $G$-action. 

In a different but related direction, several authors have considered ways to directly compute topological complexity in a discrete setting. In \cite{gonzalez2018}, Golz\`{a}lez shows that the topological complexity of the realization of a simplicial complex $K$ can be computed in terms of simplicial maps defined on subdivisions of $K$, using an invariant he defined as \textit{simplicial complexity}, denoted $\SC(X)$. In a similar vein, in \cite{tanaka_2018}, Tanaka defines a notion of \textit{combinatorial complexity} for a space $X$, denoted $\CCC_m{(X)}$, which is a descending sequence of invariants that equals $\TC{(X)}$ at the limit. Tanaka relates Gonz\`{a}lez's simplicial approach to his combinatorial approach, thus showing that the topological complexity can be entirely computed in combinatorial terms.

In this paper, we extend upon Colman and Grant's work as well as upon Tanaka and Gonz\`{a}lez's work to define \textit{equivariant combinatorial complexity} of a finite $T_0$ $G$-space and \textit{equivariant simplicial complexity} of a finite $G$-simplicial complex. We begin by following Tanaka in defining an integer sequence $\CCC_{G,n}(X)$ depending on a finite $T_0$ $G$-space $X$ and a natural number $n$. These can be thought of as versions of Colman and Grant's equivariant topological complexity where we limit the length of the paths we choose to be less than or equal to $n$. We show that $\CCC_{G,n}(X)$ decreases as $n$ goes to infinity, and define the \textit{equivariant combinatorial complexity} to be the minimum of $\CCC_{G,n}(X)$, which we denote $\CCC_G(X)$. Our first main result is that, as in the non-equivariant case, this procedure computes the equivariant topological complexity of $X$. More precisely:

\begin{thm*}[Theorem \ref{thm:tc_g=cc_g}]
	For any finite $T_0$ $G$-space $P$, it holds that $\TC_G(P) = \CCC_G(P)$. 
\end{thm*}

%Avery: I think it's okay to start this paragraph without giving background on simplicial complexity because we described it in the third paragraph beginning with "In a different but related direction."

We then follow Gonz\`{a}lez's work in defining an integer sequence $\SC_{G}^{b,c}(K)$ for a finite $G$-simplicial complex $K$. Again, we show that this is a descending sequence of invariants, the limit of which we denote $\SC_G(K)$. Our second main result is that, as in the non-equivariant case shown by Gonz\`{a}lez, the equivariant topological complexity of the geometric realization of an ordered $G$-simplicial complex $K$ is equal to the equivariant simplicial complexity of $K$.

\begin{thm*}[Theorem \ref{thm:SCeqTC}]
	Suppose that $K$ is any ordered $G$-simplicial complex. Then $\TC_G(|K|) = \SC_G(K)$. 
\end{thm*}

The remainder of the paper is organized as follows: in Section~\ref{sec:Prelims}, we begin with a background of finite $T_0$ topological spaces and their relation to finite partially ordered sets, or posets. We then provide a background on barycentric subdivision, symplicial complexes, group actions, and topological complexity.
In Section~\ref{sec:LS_G and TC_G}, we review the definitions of equivariant Lusternik-Schnirelmann category and equivariant topological complexity from \cite{colman_grant_2012}. We also recall some basic properties of these invariants from \cite{colman_grant_2012}, and prove several further properties for finite $T_0$ $G$-spaces.
In Section~\ref{sec:CC_G}, we define equivariant combinatorial complexity, we prove that the equivariant combinatorial complexity is equal to the equivariant topological complexity, and we compute the lower bound for $n$ such that $\CCC_{G,n}(X) = \CCC_G(X)$.
In Section~\ref{sec:SC_G}, we review and extend upon several properties of $G$-simplicial complexes introduced in Section~\ref{sec:Prelims}, we define equivariant simplicial complexity, and we prove that the equivariant simplicial complexity equals the equivariant topological complexity of the realization of a simplicial complex $K$.

\subsection*{Acknowledgements}
 This work was completed during the UVA REU in Topology. The authors thank Walker Stern, Matthew Feller, and Shunyu Wan for their support and guidance.  

\section{Preliminaries} \label{sec:Prelims}

\subsection{Finite Posets and Finite Topologies}

A \textit{finite topological space X} is a topological space such that the underlying set $X$ is finite. If the topology on $X$ has the property such that for any $x,y\in X$, there exists some open set $U\subseteq X$ that contains precisely one of $x$ or $y$, then this space is called a $T_0$ space, or \textit{Kolmogorov} space.

We begin by briefly recalling the connections between finite $T_0$ spaces as posets.

%\begin{definition}
%Let $P$ be any set equipped with a relation $\leq$ satisfying
%\begin{itemize}
%    \item Reflexivity: For any $a \in P$, $a \leq a$.
%    \item Antisymmetry: For any $a, b \in P$, if $a \leq b$ and $b \leq a$, then $a = b$.
%    \item Transitivity: For any $a, b, c \in P$, if $a \leq b$ and $b \leq c$, then $a \leq c$.
%\end{itemize}
%Then $P$ is called a {\it partially ordered set}, or {\it poset}. If $P$ is finite, we say that $P$ is a {\it finite poset}.

%Lastly, for finite posets $P$ and $Q$, a map $f: P \to Q$ is {\it monotone} if for every $a, b \in P$, we have $a \leq b$ implies $f(a) \leq f(b)$. 
%\end{definition}

\begin{remark}
We often draw diagrammatic sketches of posets, drawing an arrow $a\rightarrow b$ when $a$ is less than $b$. In such diagrams, we omit arrows when they can be obtained via transitivity from arrows which have already been drawn. An example is shown in figure~\ref{fig:powerset_poset}. Note that the sets $\{1\}$ and $\{3\}$ are incomparable, and that $\varnothing \subseteq \{1,2,3\}$ but the arrow is not directly shown since it can be assumed from transitivity. We have chosen these more categorically-flavored diagrams rather than, e.g., Hasse diagrams, to emphasize the connection between finite spaces and simplicial complexes.

\end{remark}
 
\begin{figure}
    \centering
    \[\begin{tikzcd}[ampersand replacement=\&]
	\& {\{1, 2, 3\}} \\
	{\{1, 2\}} \& {\{1, 3\}} \& {\{2, 3\}} \\
	{\{1\}} \& {\{2\}} \& {\{3\}} \\
	\& \varnothing
	\arrow[from=4-2, to=3-1]
	\arrow[from=4-2, to=3-2]
	\arrow[from=4-2, to=3-3]
	\arrow[from=3-1, to=2-1]
	\arrow[from=3-1, to=2-2]
	\arrow[from=3-3, to=2-2]
	\arrow[from=3-3, to=2-3]
	\arrow[from=2-3, to=1-2]
	\arrow[from=2-2, to=1-2]
	\arrow[from=2-1, to=1-2]
	\arrow[from=3-2, to=2-3, crossing over]
	\arrow[from=3-2, to=2-1, crossing over]
    \end{tikzcd}\]
    \caption{A poset diagram for $\bb{P}(\{1, 2, 3\})$.}
    \label{fig:powerset_poset}
\end{figure}
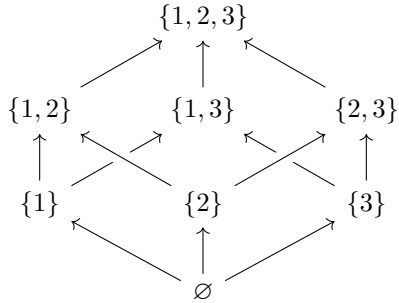

%\begin{example} \label{ex:powset}
%Let $P = \{1 < 2 < 3\}$ and $Q = \{1 < 2\}$. The map $f: P \to Q$ given by $f(1) = 1$, $f(2) = 1$, $f(3) = 2$ is monotone. The map $g: P \to Q$ given by $g(1) = 2$, $g(2) = 1$, $g(3) = 2$ is not monotone. 
%\end{example}

If X is a finite $T_0$ topological space, we can construct a poset, $P(X)$ from this space by defining an order relation on $X$ via the following procedure: if $x,y\in X$, then let $x\geq y$ if and only if $x\in \overline{\{y\}}$, where $\overline{\{y\}}$ is the closure of the singleton $\{y\}$. It can be shown that $P(X)$ is a poset.

Similarly, we can construct a finite $T_0$ topological space $T(Q)$ from a finite poset $Q$. We define the subsets $U\subseteq Q$ to be \textit{downwards-closed} if for every $y\in Q$, if there is an $x\in Q$ such that $y\leq x$, then $y\in Q$. The open sets in $T(Q)$ are the precisely the downwards-closed sets of $Q$. The resulting topological space $T(Q)$ is a finite $T_0$ space. 

%In fact, this correspondence is even stronger -- it holds in general that $T(P(X)) = X$ for any finite $T_0$ topological space $X$, and that $P(T(Q)) = Q$ for any finite poset $Q$. We also have that any continuous map $X \to Y$ is monotone when considered as a map $P(X) \to P(Y)$, and any monotone map $Q \to R$ is continuous when considered as a map $T(Q) \to T(R)$. Since any finite poset may be uniquely turned into a finite $T_0$ space this way, and vice-versa, and continuous maps between them correspond to monotone maps, we may think of finite posets and finite $T_0$ spaces as the same type of object.

These constructions yield an equivalence of categories between the category of finite $T_0$ spaces and the category of finite posets.

\begin{remark}
For the duration of this paper, we will use the terms ``finite poset'' and ``finite $T_0$ space'' interchangeably, as they essentially refer to the same object. The order on a finite $T_0$ space $X$ will always be the one induced by the topology, and the topology on a poset $P$ will always be the one induced by its order.
\end{remark}

\begin{definition}
Let $P$ be a finite poset and $x \in P$. The {\it downwards-closure} of $x$, denoted $U_x$, is the set of all $y \in P$ such that $y \leq x$. It can be shown that $U_x$ is equivalently the smallest open set containing $x$, i.e., the intersection of all open sets containing $x$.
\end{definition}

\begin{example}
The \textit{combinatorial interval} ($J_m$) of length $m+1$ is the poset 
\[
0<1>2<...>(<)m.
\]   

The open sets on the topology associated to $J_m$ are the downwards closed sets. For example, the open sets on the topological space associated to $J_3$ are shown in figure \ref{fig:comb_int_open_sets}.
\end{example}

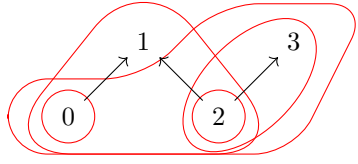
\begin{figure}
    \centering
    \begin{tikzpicture}
\node (0) at (0, 0) {$0$};
\node (1) at (1, 1) {$1$};
\node (2) at (2, 0) {$2$};
\node (3) at (3, 1) {$3$};

\draw[red, rounded corners=30pt] (-1, -0.5) -- (1, 2) -- (3, -0.5) -- cycle; 
\draw[rotate=45, red] (2, -1.41) ellipse (1.1 and 0.6);
\draw[red] (0) circle (10pt);
\draw[red] (2) circle (10pt);
\draw[red, rounded corners=15pt] (4, 1.5) -- (2, 1.5) -- (1, 0.5) -- (-0.8, 0.5) -- (-0.8, -0.5) -- (3, -0.5) -- cycle;

\draw[->] (0) -- (1);
\draw[->] (2) -- (1);
\draw[->] (2) -- (3);
\end{tikzpicture}
    \caption{The combinatorial interval $J_3$ with all its open sets highlighted in red except for $J_3$ itself and $\varnothing$}
    \label{fig:comb_int_open_sets}
\end{figure}

We observe that there is a surjective continuous map $f: I\to J_m$ defined by 
\[
f(t):=\begin{cases}
	2i+1 &  t=\frac{2i+1}{m}; \quad i\in \ZZ \\
	2i & t\in\left(\frac{2i-1}{m},\frac{2i+1}{m}\right); \quad i\in \ZZ 
\end{cases}
\]
Per \cite{McCord}, this map is a weak homotopy equivalence. As such we can think of $J_m$ as a finite analogue of the interval.

\begin{definition}
Using $J_m$ as a finite analogue of the interval $I$, we can define a \textit{combinatorial path} from $x$ to $y$ to be a map $\beta: J_m\rightarrow X$ such that $\beta(0) = x$ and $\beta(m) = y$. 

Equipped with $J_m$, we can defined \textit{combinatorial homotopy}. Let $Q$ and $R$ be finite posets and $f,g:Q\rightarrow R$ be monotone maps. A combinatorial homotopy between $f$ and $g$ is a map \[h:Q\times J_m\rightarrow R\] such that $h|_{Q\times \{0\}}=f$ and $h|_{Q\times \{m\}}=g$.

\end{definition}

There exists a strong correlation between combinatorial paths and topological paths, as the following proposition shows.

\begin{proposition}
Let $X$ be a topological space and let $x,y\in X$. There is a topological path from $x$ to $y$ if and only if there is a combinatorial path from $x$ to $y$.
\end{proposition}
\begin{proof}
Due to Barmak \cite[Prop. 1.2.4]{barmak_2011}
\end{proof}

Similarly, using $J_m$, we can define a homotopy between maps of posets.

\begin{remark}
Suppose $X$ and $Y$ are finite $T_0$ spaces. We denote $C(X,Y)$ to be the set of all continuous maps $Y \to X$. We can define a partial order on this set by requiring that for any $f, g: Y \to X$, $f \leq g$ if and only if $f(y) \leq g(y)$ for every $y \in Y$. We may induce a topology on $C(X,Y)$ via the order $\leq$ which conveniently coincides with the compact-open topology (see \cite[Prop. 1.2.5]{barmak_2011}). We call this space equipped with the compact-open topology $X^Y$. There is thus no ambiguity when we write $X^{J_m}$.

\end{remark}

\begin{definition}
We can replace the unit interval $I=[0,1]$ with the combinatorial interval $J_m$. We can then denote the \textit{combinatorial path space} by $X^{J_m}$. 
By construction, there is a continuous map $q_m:X^{J_m}\rightarrow X\times X$ which sends a path $\alpha$ in the path space $X^{J_m}$ to its endpoints $(\alpha(0),\alpha(m))$, called the \textit{combinatorial path fibration}.
\end{definition}

We now recall the notion of the \textit{subdivision} of a poset.
\begin{definition}
Suppose $P$ is a finite poset. Then the {\it barycentric subdivision} $\sd(P)$ is the set of totally ordered subsets of $P$ ordered under the relation $\subseteq$. If we subdivide the same space $P$ a total of $n$ times, we denote the resulting space $\sd^n(P)$. 
\end{definition}
It turns out that many invariants we will consider behave nicely under subdivision. For example, we shall see later that the equivariant topological complexity of $P$ is greater than or equal to the equivariant topological complexity of $\sd(P)$. Pictorially, $\sd(P)$ looks similar to $P$ itself, as can be seen in figure \ref{fig:PosetSubdivision}. 

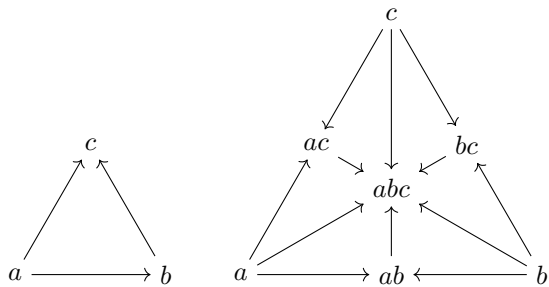
\begin{figure}
    \centering
    \begin{tikzpicture}
\node (a) at (-5, 0) {$a$};
\node (b) at (-3, 0)  {$b$};
\node (c) at (-4, 1.73) {$c$};

\draw[black, ->] (a) -- (b);
\draw[black, ->] (a) -- (c);
\draw[black, ->] (b) -- (c);

\node (a) at (-2, 0) {$a$};
\node (b) at (2, 0)  {$b$};
\node (c) at (0, 3.464) {$c$};
\node (ab) at (0, 0) {$ab$};
\node (ac) at (-1, 1.732) {$ac$};
\node (bc) at (1, 1.732) {$bc$};
\node (abc) at (0, 1.155) {$abc$};

\draw[black, ->] (ab) -- (abc);
\draw[black, ->] (ac) -- (abc);
\draw[black, ->] (bc) -- (abc);
\draw[black, ->] (a) -- (abc);
\draw[black, ->] (b) -- (abc);
\draw[black, ->] (c) -- (abc);
\draw[black, ->] (a) -- (ab);
\draw[black, ->] (b) -- (ab);
\draw[black, ->] (a) -- (ac);
\draw[black, ->] (c) -- (ac);
\draw[black, ->] (b) -- (bc);
\draw[black, ->] (c) -- (bc);
\end{tikzpicture}
    \caption{The poset $\{a < b < c\}$ and its subdivision}
    \label{fig:PosetSubdivision}
\end{figure}

\begin{definition}
We denote by $\tau_{P}:\sd(P)\to P$ the monotone \textit{last-vertex map} defined by $\{p_0<p_1<\dots<p_n\} \mapsto p_n$. When no confusion is likely to arise, we we will abuse notation by simply denoting $\tau_P$ by $\tau$.
\end{definition}

%\begin{remark} 
%The map $\tau: \sd(P) \to P$ that maps $\{p_0<p_1<\dots<p_n\} \mapsto p_n$ is a canonical map from the subdivision to its original space that is monotone and continuous.
%\end{remark}
\subsection{Simplicial Complexes}

Another common way of describing topological spaces
by a finite amount of combinatorial data is via \textit{simplicial complexes}. We will use the following definition of simplicial complexes.

\begin{definition}
A \textit{simplicial complex} $K$ is a finite set of vertices, $X$ with a collection of subsets $\sf{Sim}(K) \subseteq \bb{P}(X) \setminus \varnothing$. Such that 
\begin{itemize}
    \item All singletons from $X$ are in $\sf{Sim}(K)$, and
    \item Given $Y \in \sf{Sim}(K)$ and $Z \subseteq Y$, then $Z \in \sf{Sim}(K)$.
\end{itemize}
An element of $\sf{Sim}(K)$ is called a \textit{simplex} of $K$. And given simplices $Z$ and $Y$, if $Z \subseteq Y$, we say $Z$ is a \textit{face} of $Y$.
\end{definition}
In the last section of the paper, we will provide a characterization of the equivariant topological complexity of realizations of simplicial complexes.
\begin{definition}
The \textit{geometric realization} of a simplicial complex $K$ with vertex set $X$, denoted $|K|$ is a topological space consisting of the formal linear combinations
\[ \Lambda = \sum_{x \in X} \lambda_x \cdot x \]
where $\lambda_x \in \RR$. For any such formal linear combination, we define a simplex $S_\Lambda := \{x \in X \mid \lambda_x \neq 0\}$. Finally, we can define $|K|$ as follows
\[|K| := \left \{ \sum_{x \in X} \lambda_x \cdot x \mid \lambda_x \geq 0, \sum\lambda_x = 1, S_\Lambda \in \sf{Sim}(K)\right \} \subseteq \RR^{\#(X)} \]
$|K|$ inherits the subspace topology from $\RR^n$. 
\end{definition}

\begin{remark}
Because we use $|\cdot|$ to denote \emph{geometric realization}, we denote the cardinality of a set $X$ by $\#(X)$.
\end{remark}

\begin{example}
The \textit{standard combinatorial $n$-simplex}, $\Delta^n$ is defined by the set $X = \{0, 1, 2, \dots, n\}$ and $\sf{Sim}(\Delta^n) = \bb{P}(X) \setminus \{\varnothing\}$.
\end{example}

We can also make simplicial complexes out of posets.

\begin{definition}
Let $P$ be a finite poset. The {\it order complex} of $P$, denoted $\scr{K}(P)$, is the simplicial complex whose vertex set is $P$, and set of simplices is
\[
\Sim(\scr{K}(P)) := \{Q \subseteq P \mid Q \text{ is totally ordered}\}
\]
\end{definition}

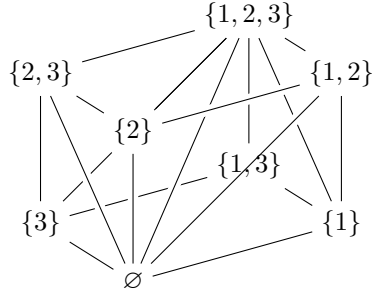
\begin{figure}
    \centering
    \begin{tikzpicture}[scale=2]
    \node (empty) at (0, 0, 0) {$\varnothing$};
    \node (1) at (1, 0, -1) {$\{1\}$};
   
    \node (3) at (-1, 0, -1) {$\{3\}$};
    \node (13) at (0, 0, -2) {$\{1, 3\}$};
    \node (23) at (-1, 1, -1) {$\{2, 3\}$};
    \node (12) at (1, 1, -1) {$\{1, 2\}$};
    \node (123) at (0, 1, -2) {$\{1, 2, 3\}$};
    \draw (123) -- (3);
     \path[fill=white] (0,1,0) circle (0.2); %line passes behind {2}
    \node (2) at (0, 1, 0) {$\{2\}$};
    \draw (123) -- (23);
    \draw (123) -- (1);
    \draw (123) -- (2);
   
    \draw (13) -- (3);
    \draw[white, line width=4pt] (123) -- (empty);
    \draw (123) -- (empty);
    \draw (123) -- (12);
    \draw (123) -- (13);
    \draw (13) -- (1);
    \draw (13) -- (empty);
    \draw[white, line width=4pt] (12) -- (empty);
    \node at (0, 0, -2) {\colorbox{white}{$\{1, 3\}$}};
    \draw (12) -- (empty);
    \draw[white, line width=4pt] (23) -- (empty);
    \draw (23) -- (empty);
    \draw (23) -- (3);
    \draw (23) -- (2);
    \draw (12) -- (1);
    \draw[white, line width=4pt] (12) -- (2);
    \draw (12) -- (2);
    \draw[white, line width=4pt] (2) -- (empty);
    \draw (2) -- (empty);
    \draw (3) -- (empty);
    \draw (1) -- (empty);
    \end{tikzpicture}
    \caption{The order complex $\scr{K}(\bb{P}(\{1, 2, 3\})$ associated to a power set.}
    \label{fig:my_label}
\end{figure}

The subdivision construction is more commonly applied to simplicial complexes.

\begin{definition}
Suppose $K$ is a simplicial complex. Define $\chi(K) := \Sim(K)$ and equip it with the order $\subseteq$. We call $\chi(K)$ the {\it face poset} of $K$. Define the \textit{barycentric subdivision} of $K$ as
\[
\sd(K) := \scr{K}(\chi(K)).
\]
\end{definition}

\begin{remark}
Using the idea of the face poset, we may redefine subdivision of posets in a similar way to subdivision of simplicial complexes. For any poset $P$, $\sd(P) = \chi(\scr{K}(P))$. This serves to justify our abuse of notation in using $\sd$ to denote both the subdivision of a poset and the subdivision of a simplex.
\end{remark}

For the reader's convenience, we recall the following well-known fact about subdivision without proof.

\begin{proposition}
For any simplicial complex $K$, there exists a homeomorphism $J: |\sd(K)| \to |K|$. Thus,
\[
|\sd(K)| \cong |K|
\]
\end{proposition}

In the sequel we will identify the realizations of simplicial complexes and their subdivisions by this homeomorphism without comment.

\begin{definition}
Given simplicial complexes $K$ and $L$, a \textit{simplicial map} $\phi : K \to L$ is a map between the vertex sets of $K$ and $L$ such that $\phi(\sigma) \in \sf{Sim}(L)$ for any simplex $\sigma \in \sf{Sim}(K)$.
\end{definition}

As in \cite{gonzalez2018}, our definition of equivariant simplicial complexity will require us to make use of simplicial approximations to discretize homotopy-theoretic constructions. We recall some of these constructions here, and elaborate on them in Section~\ref{sec:SC_G}.

\begin{definition}
Given a map $f: |K| \to |L|$, we say that a simplical map $\phi: K \to L$ is a \textit{simplicial approximation} of $f$ if for $x \in |K|$ and $\sigma \in \sf{Sim}(L)$ and $f(x) \in |\sigma|$ implies that $|\phi |(x) \in |\sigma|$.
\end{definition} % I think it might be better to use the second definition here so that we don't have to define what an open simplex is.

This definition gives us a means of relating simplicial maps to topological maps, but not to each other. For that we turn to the notion of \textit{contiguity} used in \cite[Definition 2.5]{gonzalez2018}

\begin{definition}
Given two simplicial maps $\phi,\psi: K \to L$, we say that $\phi$ and $\psi$ are \textit{$1$-contiguous} if $\phi(\sigma)\cup\psi(\sigma)$ is a simplex of $L$ for any simplex $\sigma$ of $K$.

And that, given a positive integer $l$, $\phi$ and $\psi$ are $l$-contiguous if there exists a sequence of maps $\phi = \phi_0,\phi_1, \dots, \phi_l = \psi: K \to L$ such that $\phi_i$ and $\phi_{i+1}$ are $1$-contiguous.
\end{definition}

In the original definition of simplicial complexity in \cite{gonzalez2018}, heavy use is made of the following properties of simplicial approximations, which we recapitulate from \cite{gonzalez2018}. We will use both these properties and their equivariant analogues heavily in the sequel.

\begin{theorem} \label{thm:contiguityfacts}
The following facts hold regarding simplicial approximations \begin{itemize}
    \item Two approximations of the same continuous map are $1$-contiguous, and their contiguity class is unique.
    \item Simplicial maps in the same contiguity class have homotopic realizations.
    \item Given homotopic maps $f, g: |K| \to |L|$, there exists a non-negative integer $n_0$ such that for any $n \geq n_0$, any pair of approximations $\phi, \psi: \sd^n(K) \to L$ of $f$ and $g$ respectively are $c$-contiguous for some $c \geq 0$.
\end{itemize}
\end{theorem}

\subsection{Group Actions} 

We begin by recalling some facts about group actions and continuous group actions on topological spaces. 

\begin{remark}
For the purposes of the rest of this paper, we will assume that any group $G$ is a finite group (equipped with the discrete topology when applicable).
\end{remark}

Suppose that a group $G$ acts continuously on a space $X$. For any $g \in G$ and $x \in X$, we denote $g \cdot x$ the action of $g$ on $x$. We denote $G_x$ the {\it stabilizer group} or {\it isotropy group} of $x$. We denote $\cal{O}(x)$ the orbit of $x$ and $X / G$ the space of orbits of $X$ (endowed with the quotient topology). We write $X^G$ for the $G$-fixed point set of $X$, which we equip with the subspace topology. If $X$ and $Y$ are $G$-spaces, we let $G$ act diagonally on the product space $X \times Y$. If $V \subseteq X$ is any subset, denote $G \cdot V := \{g \cdot v \mid g \in G, v \in V\}$. One may easily see that $G \cdot V$ is always an invariant subset of $X$.

\begin{definition}
Suppose $X$ and $Y$ are $G$-spaces, $f, g: X \to Y$ are continuous, equivariant maps. We can consider $X \times I$ as a $G$-space by letting $G$ act diagonally on $X \times I$ and trivially on the unit interval $I$. We say that a continuous map $H: X \times I \to Y$ is a {\it $G$-homotopy} from $f$ to $g$ if $H$ is a homotopy from $f$ to $g$ and $H$ is equivariant. If there exists a $G$-homotopy between $f$ and $g$, we say that $f$ is {\it $G$-homotopic} to $g$ and write $f \simeq_G g$.
\end{definition}

Now we must revisit the notion of subdivision and extend it to the setting of $G$-spaces.

\begin{definition}
For a finite $T_0$ $G$-space $X$ with a subdivision $\sd^n(X)$, we define a \textit{G-action on the barycentric subdivision} by $\{g\cdot p_0 < g\cdot p_1 \dots <g\cdot p_n\} = g\cdot\{p_0<p_1< \dots <p_n\}$. 
\end{definition}

\begin{lemma}\label{lem:EquivarT}
The canonical map $\tau: \sd^n(X) \to X$ defined by $\tau: \{p_0<p_1<\dots<p_n\} \mapsto p_n$ is $G$-equivariant.
\end{lemma}

\subsection{Topological Complexity}
We now recall some key facts about topological complexity, mostly following \cite{farber_2003} and \cite{farberbook}.
If $X$ is a topological space, we define $X^I$ to be the set of all paths in $X$, called the \textit{path space}, and endow it with the compact-open topology. The \textit{path fibration} is the mapping $p:X^I\rightarrow X\times X$ from the path space to $X \times X$, That is
\begin{align*}
    p: X^I &\to X \times X \\
    \alpha &\mapsto (\alpha(0), \alpha(1))
\end{align*}
Given any two points $a,b\in X$, we want to choose a path $\alpha$ between $a$ and $b$. We do this using the idea of a motion planning.
\begin{definition}
Given any two points $a,b\in X$, a \textit{motion planning} is a continuous map $s: X \times X \to X^I$ such that $p \circ s = \sf{id}_{X^I}$
\end{definition}

\begin{remark} \label{rem: tc=1 means contractible}
For an arbitrary topological space $X$, it is rarely the case that a continuous motion planning can be defined on the entire space. In fact, if $s:X\times X\rightarrow X^I$ is a continuous motion planning for $X$, then $X$ is contractible. 
\end{remark}

In spite of this, we can consider motion plannings on smaller pieces of $X \times X$.

\begin{definition}
Suppose $U \subseteq X \times X$ is any set. If there exists a continuous map $s$ such that the following diagram commutes:
% https://q.uiver.app/?q=WzAsMyxbMCwxLCJVIl0sWzEsMCwiWF5JIl0sWzEsMSwiWCBcXHRpbWVzIFgiXSxbMCwxLCJzIl0sWzEsMiwicCJdLFswLDIsIiIsMix7InN0eWxlIjp7InRhaWwiOnsibmFtZSI6Imhvb2siLCJzaWRlIjoidG9wIn19fV1d
\[\begin{tikzcd}
	& {X^I} \\
	U & {X \times X}
	\arrow["s", from=2-1, to=1-2]
	\arrow["\pi", from=1-2, to=2-2]
	\arrow[hook, from=2-1, to=2-2]
\end{tikzcd}\]
then we say that $U$ {\it admits a continuous motion planning}. An open set $U \subseteq X \times X$ that admits a continuous motion planning is called a \textit{sectional categorical} set.
\end{definition}

\textit{Topological Complexity}, as defined by Farber in \cite{farber_2003} is a homotopy invariant that formalizes the idea of finding a motion planning for \textit{any} arbitrary space $X$ by considering separate continuous motion plannings on open sets of $X \times X$.

\begin{definition}\label{def:TC}
Let $X$ be a topological space. The topological complexity of $X$, denoted $\TC{(X)}$, is the minimal number $k$ such that $X\times X$ can be covered by open subsets $U_1, U_2, ..., U_k\subseteq X\times X$ where each $U_i$ admits a continuous motion planning. 
\end{definition}

As discussed in Remark \ref{rem: tc=1 means contractible}, topological complexity is closely related to contractiblity. Thus it should not be surprising that topological complexity is closely related to the \textit{Lusternik-Schnirelmann category of $X$}. 

\begin{definition}
 Let $X$ be a topological space. A set $U \subseteq X$ is called {\it nullhomotopic} if the inclusion $i_U: U \to X$ is homotopic to a constant map $U \to X$. The \textit{Lusternik-Schnirelmann category} of $X$, denoted $\LS{(X)}$, is the minimal number $k$ such that $X$ can be covered by nullhomotopic open sets $U_1, U_2, ..., U_k\subseteq X$.
\end{definition}

The Lusternik-Schnirelmann category is valuable in that it provides several bounds on the topological complexity of a space.  
 
\begin{proposition}\label{prop:TC_LS_Bounds}$\;$\\
\begin{enumerate}
    \item If $X$ and $Y$ are paracompact and path-connected topological spaces, then \[\LS{(X\times Y)}<\LS{(X)}+\LS{(Y)}\]
    \item If $X$ is a path-connected topological space, then \[\LS{(X)}\leq \TC{(X)}\leq \LS{(X\times X)}\]
\end{enumerate}

\end{proposition}
\begin{proof}
The first statement is proved in \cite[Proposition 2.3]{JAMES1978331}, and the second is proved in \cite[Theorem 5]{farber_2003}.
\end{proof}

\section{Equivariant LS-Category and Equivariant Topological Complexity}\label{sec:LS_G and TC_G}

In their 2012 paper \cite{colman_grant_2012}, Colman and Grant define an equivariant version of Farber's topological complexity, including its properties and product bounds. The basis of their definition lies in the sectional category.

\begin{definition}
The \textit{sectional category} of a map $p:E\rightarrow B$, denoted $\secat{(p)}$, is the least integer $k$ such that $B$ may be covered by $k$ open sets, $U_1, U_2, ..., U_k\subseteq B$, on each of which there exists a map $s: U_i\rightarrow E$ such that $p\circ s: U_i\rightarrow B$ is homotopic to the inclusion $i_{U_{i}}:U_i\rightarrow B$. If no such integer exists, we let $\secat{(p)}=\infty$.
\end{definition}

\begin{remark}
If the map $p: E \to B$ is a fibration, one may require that $p\circ s = i_{U_i}$ instead of requiring that $p \circ s \simeq_G i_{U_i}$, and the value of $\secat_G(p)$ will be unchanged. In particular, the path fibration $\pi: X^I \to X \times X$ and the combinatorial path fibration $q_m: P^{J_m} \to P \times P$ are fibrations.
\end{remark}

Equipped with this definition, we can rephrase our original definition for topological complexity (Definition~\ref{def:TC}) in terms of the sectional category.
Recall that the path fibration is a mapping from $p:X^I\rightarrow X\times X$ from a path to its endpoints. In addition, we can recall that $\TC{(X)}$ is the minimal number $k$ such that $X\times X$ can be covered by open subsets, each of which admit a continuous motion planning. Therefore, it is clear that the topological complexity of a space can be redefined as the sectional category of the path fibration. \[\TC{(X)}=\secat{(\pi:X^I\rightarrow X\times X)}\]

In \cite{colman_grant_2012}, Colman and Grant use this rephrased definition to develop equivariant topological complexity. They define equivariant versions of the Lusternik-Schnirelmann category and sectional category, and then apply these to the equivariant path fibration.

\begin{definition}
The \textit{equivariant category of a G-space}, $\LS_G(X)$ is the least integer $k$ such that $X$ may be covered by $k$ open, invariant sets, each of which has an inclusion function $i_{U_i}:U_i\rightarrow X$ which is G-homotopic to to a map with values in a single orbit. We call such subsets \textit{G-categorical}
\end{definition}

In a loose sense the above definition allows us to determine how equivariantly simple a space is. Next we introduce a notion of that simplicity with respect to an equivariant map.

\begin{definition}
The \textit{equivariant sectional category} of a $G$-map $p: E \to B$, denoted $\secat_G(p)$ is the least $k$ in the integers such that $B$ is covered by $k$ open invariant sets such that for each $U_i$ there exists a $G$-map $s_i: U_i \to E$ where $p \circ s_i = i_{U_i}$.
\end{definition}

\begin{remark}
We can observe that if $G$ acts trivially on $X$, then $\LS_G(X)=\LS(X)$ and $\secat_G(p)=\secat(p)$, as expected. 
\end{remark}
 
\begin{remark}
If $X$ is a G-space, $\pi: X^I\rightarrow X\times X$ is a G-fibration with respect to the actions. We can verify that the free path fibration $\pi$ is $G$-equivariant. Let $g \in G$ and let $\alpha \in X^I$. Since $X$ is a $G$-space, 
$$
\pi(g(\alpha(t))) = (g(\alpha(0)), g(\alpha(1))) = g(\alpha(0), \alpha(1)) = g(\pi(\alpha(t))).
$$
\end{remark}
Using this path fibration, Colman and Grant define an equivariant version of topological complexity:

\begin{definition}
The \textit{equivariant topological complexity} is defined as the equivariant sectional category of the free path fibration $\pi:X^I\rightarrow X\times X$. I.e., \[\TC_G(X) = \secat_G(\pi)\]
\end{definition}

\subsection{Properties}

Colman and Grant prove several bounds on the product and properties of $\TC_G(X)$ which prove useful to our later work. While Colman and Grant only consider $G$-spaces which are Hausdorff, many of their arguments carry over verbatim to the finite $T_0$ setting. 

\begin{definition}
We say a $G$-space $X$ is {\it $G$-connected} if the fixed point set $X^H$ is path-connected for every closed subgroup $H \subseteq G$. Since we equip $G$ with the discrete topology in this paper, then $X$ is $G$-connected if $X^H$ is path-connected for {\it any} subgroup $H \subseteq G$. 
\end{definition}

It turns out that $G$-connectedness is sufficient to generalize some of the bounds on $\TC(X)$ to the equivariant setting.

\begin{lemma}\label{prop:GConnectedBound}
If $X$ is $G$-connected, then $\TC_G(X)\leq \LS_G(X\times X)$ 
\end{lemma}
\begin{proof}
This is \cite[Prop 5.6]{colman_grant_2012}.
\end{proof}

Now we turn to an equivariant generalization of part (1) of Proposition~\ref{prop:TC_LS_Bounds}.

\begin{lemma}
If $X$ is $G$-connected and $X^G \neq \varnothing$, then $\LS_G(X) \leq \TC_G(X)$.
\end{lemma}
\begin{proof}
This follows from \cite[Prop 5.7]{colman_grant_2012}.
\end{proof}

The following proposition proves useful in fixing the lengths of paths that we map to by a particular equivariant sectional category called an equivariant combinatorial motion planning in Proposition \ref{prop:CCpathlength}.

\begin{proposition}\label{prop:PathLength}
Any two points $x, y$ in a path-connected finite space $X$ may be connected using a path $\alpha: J^{n} \to X$ where $n = \#(X)$.

\begin{proof}
  We proceed via induction on $\#(X)$. If $\#(X) = 1$, then $X$ is connected. If $\#(X) = 2$, then for $x \le y \in X$ there exists a path $\gamma: J^2 \rightarrow X$ defined by 
  $$
  \gamma(j) = \begin{cases}
  x &\text{ if } j = 0,
  \\
  y &\text{ if } j = 1,2.
  \end{cases}
  $$
  
  Suppose that the proposition holds for $\#(X) \leq n$ and that $\#(X) = n+1$. Let $x, y \in X$. Since $X$ is path-connected, there exists some $z \in X$ such that $z$ is related to $y$ and there is a path $\beta: J^n \rightarrow X$ such that $\beta(0) = x$ and $\beta(n) = z$. We can define a path
  $$
  \alpha(j) = \begin{cases}
  \beta(j) &\text{ if } j < n,
  \\
  z &\text{ if } j = n \text{ and either } \beta(n-1) \le z \ge y \text{ or } \beta(n-1) \ge z \le y,
  \\
  y &\text{ if } j = n \text{ and either } \beta(n-1) \le z \le y \text{ or } \beta(n-1) \ge z \ge y,
  \\
  y &\text{ if } j = n+1
  \end{cases}
  $$
  that maps $J^{n+1}$ to $X$.

  %Then there is a combinatorial path $\beta: J_m \to X$ such that $\beta(0) = x$ and $\beta(m) = y$, for some $m \in \NN$.
  
  %Since we may truncate combinatorial paths, we can assume without loss of generality that $\beta$ only takes the value $y$ at $\beta(m)$. Thus, the space $B := \beta(J_{m-1})$ is a subset of $X \setminus \{y\}$. Hence, $\#(B) \leq n$. Since two subsequent elements of a combinatorial path are always comparable, we have two cases: either $\beta(m-1) \leq \beta(m)$ or $\beta(m-1) \geq \beta(m)$. 
  
  %Assume $\beta(m-1) \geq \beta(n)$. Since $B$ is path-connected and $\#(B) \leq n$, we may apply the inductive hypothesis to find a length $n$ path between any two points of $B$. In particular, $x \in B$ and $\beta(m-1) \in B$, so there is a path $\alpha: J_n \to B$ such that $\alpha(0) = x$ and $\alpha(n) = \beta(m-1)$. If $\alpha(n-1) \leq \alpha(n)$, construct the path
  %\[
  %x = \alpha(0) \leq \alpha(1) \geq \alpha(2) \leq \dots \geq \alpha(n-1) \leq \alpha(n) = \beta(m-1) \geq \beta(m) 
  %\]
  %of length $n+1$. If $\alpha(n-1) \geq \alpha(n)$, then $\alpha(n-1) \geq \beta(m-1) \geq  \beta(m)$. We may construct the path
  %\[
  %x = \alpha(0) \leq \alpha(1) \geq \alpha(2) \leq \dots \leq \alpha(n-1) \geq \beta(m) \leq \beta(m) = y
  %\]
  %of length $n+1$. Since we may find a path of length $n+1$ in either case, the result follows by induction. The case where $\beta(m-1) \leq \beta(m)$ is similar. 
\end{proof}
\end{proposition}

\subsection{Lusternik-Schnirelmann Category of Finite Spaces}

We now turn to an equivariant generalization of part (1) of Proposition~\ref{prop:TC_LS_Bounds}.

\begin{lemma}\label{lem:TopOnOrbit}
Suppose $G$ is a finite group. If $X$ is a $T_0$ $G$-space, and $x \in X$, then equipping $\cal{O}(x)$ with the subspace topology yields the discrete topology on $\cal{O}(x)$.
\end{lemma}
\begin{proof}
Let $x \in X$ and let $g \in G$ such that $g \cdot x \not = x$. Since $g$ is monotone, we have $x \not \ge g \cdot x$ and $x \not \le g \cdot x$. It follows that the largest downwards-closed subset of $\cal O(x)$ contains only one element, concluding the proof.
\end{proof}

\begin{proposition}\label{prop:LSofProduct}
Let $G$ be a finite group and let $X$ and $Y$ be $T_0$ $G$-spaces. Then $\LS_G(X \times Y) \leq \#(G)\LS_G(X)\LS_G(Y)$. 
\end{proposition}

\begin{proof}
Suppose that $\LS_G(X) = m$ and $\LS_G(Y) = n$. Then $X$ has an open, invariant cover $\{U_i\}_{i=1}^m$ such that there exists a family of $G$-homotopies $H_i: U_i \times I \to X$ such that $H_i|_0 = i_{U_i}$ and the image $H_i|_1(U_i) \subseteq \cal{O}(x_i)$, whenever $1 \leq i \leq m$. Similarly, $Y$ has an open, invariant cover $\{V_j\}_{j=1}^m$ with a family of $G$-homotopies $K_j: V_j \times I \to Y$ such that $K_j|_0 = i_{V_j}$ and the image $K_j|_1(V_j) \subseteq \cal{O}(y_j)$.

For any $U_i$ and $V_j$ we can define a $G$-homotopy $F_{ij}: U_i \times V_j \times I \to X \times Y$ by $F_{ij}(x, y, t) = (H_i(x, t), K_j(y, t))$. This homotopy has the properties that $F_{ij}|_0 = i_{U_i} \times i_{V_j} = i_{U_i \times V_j}$ and that the image $F_{ij}|_1(U_i \times V_j) = H_i|_1(U_i) \times K_j|_1(V_j) \subseteq \cal{O}(x_i) \times \cal{O}(y_j)$. 

%Suppose $|G| = \ell$. Since $\cal{O}(x_i) \times \cal{O}(y_j)$ is invariant, we may restrict the action of $G$ on $X \times Y$ to $\cal{O}(x_i) \times \cal{O}(y_j)$. Therefore, $\cal{O}(x_i) \times \cal{O}(y_j)$ is the union of at most $|G|$ distinct orbits $O_1, O_2, \dots, O_\ell \subseteq \cal{O}(x_i) \times \cal{O}(y_j)$. By \ref{lem:TopOnOrbit}, $\cal{O}(x_i)$ and $\cal{O}(y_j)$ both have the discrete topology, meaning their product $\cal{O}(x_i) \times \cal{O}(y_j)$ also has the discrete topology. Therefore, each $O_k$ is an open, invariant set in $\cal{O}(x_i) \times \cal{O}(y_j)$. If we define $W_{ijk} = (F_{ij}|_1)^{-1}(O_k)$, then $W_{ijk}$ is open in $U_i \times V_j$ and invariant, and $\{W_{ijk}\}_{k=1}^\ell$ forms a cover of $U_i \times W_i$. Since $U_i \times V_j$ is open, and $W_{ijk}$ is open in the subspace topology $U_i \times V_j$, then by general topology, $W_{ijk}$ is open in $X \times Y$. Since the $U_i \times V_j$'s cover $X \times Y$ and the $W_{ijk}$'s cover each member of that cover, then the $W_{ijk}$'s also cover $X \times Y$.

If $g_1, g_2, \dots, g_\ell \in G$ are all the elements of $G$, then $\cal{O}(x_i) \times \cal{O}(y_j)$ can be written as a union of sets $\cal{O}((x_i, g_1 \cdot y_j)), \cal{O}((x_i, g_2 \cdot y_j)), \dots, \cal{O}((x_i, g_\ell \cdot y_j))$. By \ref{lem:TopOnOrbit}, $\cal{O}(x_i)$ and $\cal{O}(y_j)$ both have the discrete topology, meaning their product $\cal{O}(x_i) \times \cal{O}(y_j)$ also has the discrete topology. Therefore, $\cal{O}((x_i, g_k \cdot y_j))$ is an open, invariant set in $\cal{O}(x_i) \times \cal{O}(y_j)$. If we define $W_{ijk} = (F_{ij}|_1)^{-1}\left(\cal{O}((x_i, g_k \cdot y_j))\right)$, then $W_{ijk}$ is open in $U_i \times V_i$ and invariant, and $\{W_{ijk}\}_{k=1}^\ell$ forms a cover of $U_i \times W_i$. Since $U_i \times V_j$ is open, and $W_{ijk}$ is open in the subspace topology $U_i \times V_j$, then by general topology, $W_{ijk}$ is open in $X \times Y$. Since the $U_i \times V_j$'s cover $X \times Y$ and the $W_{ijk}$'s cover each member of that cover, then the $W_{ijk}$'s also cover $X \times Y$. 

Finally, we show that each $W_{ijk}$ is $G$-categorical. Recall our homotopy $F_{ij}$ from earlier and restrict it to $W_{ijk} \times I$. Then $F_{ij}|_{W_{ijk} \times \{0\}}  = i_{U_i \times V_j}|_{W_{ijk}} = i_{W_{ijk}}$. We also have by definition of $W_{ijk}$ that the image $F_{ij}|_1(W_{ijk}) = \cal{O}((x_i, g_k \cdot y_j))$. Thus, $W_{ijk}$ is $G$-categorical and so $\{W_{ijk}\}_{(i,j,k)=(1,1,1)}^{(m,n,l)}$ is an open invariant $G$-categorical cover of $X \times Y$. We see that there are exactly $l \cdot m \cdot n = \#(G)\LS_G(X)\LS_G(Y)$ sets in this cover as desired.
\end{proof}

The above proposition provides a valuable bound for combinatorial calculations. The following proposition is a first insight into the utility of subdivisions.

\begin{proposition}
Let $X$ be a finite $T_0$ $G$-space, then $\LS_G(\sd(X)) \leq \LS_G(X)$.
\end{proposition}
\begin{proof}
Let $X$ be a finite $T_0$ $G$-space with the subdivision $\sd(X)$. Say $\LS_G(X)=k$, then there exists and open invarianct cover $\{U_i\}_{i=1}^k$ and for every $i$, the inclusion $i_{U_i}: U_i \hookrightarrow X$ is $G$-homotopic to a map taking values in a single orbit. Then let $F_i: U_i \times J_m \to X$ be some $G$-homotopy between the inclusion and an orbit. So $F_i|_{U_i \times \{0\}} = i_{U_i}$ and $F_i|_{U_i \times \{m\}} \subseteq \cal{O}(z)$ where $z \in X$. Let $V_i = \sd(U_i) \subseteq \sd(X)$.
Since the canonical map $\tau$ is equivariant and $U_i$ is invariant, $V_i$ is as well. $V_i$ is also open by the topology on $\sd(X)$ and given the homotopy $F_i$, we can construct a homotopy $H_i: V_i \times J_{m+1} \to \sd(X)$ by 
\[H_i(v, j) = \begin{cases}v       &\text{ if } j=0\\
F(\tau(v), j-1)            &\text{ if } 0<j\leq m
\end{cases}\]
Since $X$ and $\sd^n(X)$ are both $G$-spaces, and $\tau$ is equivariant $H$ is also equivariant under $G$. $H_i|_{V_i \times \{0\}} = i_{v_i}$ and $H_i|_{V_i \times \{m+1\}} = F|{U_i \times \{m\}} \subseteq \cal{O}(z)$ for some $z \in X$. Since $z \in X$, $z \in \sd(X)$ as well. And, since we define the $G$-action on the subdivision by the group action on $X$, the orbit of $z$ is the same in both $X$ and $\sd(X)$. So there is an open invariant cover $\{V_i\}_{i=1}^k$ of $\sd(X)$ and $\LS_G(\sd(X)) \leq k$, as desired.
\end{proof}
\begin{remark}
It is easy to check that for any poset $Q$, $\sd(Q) = \sd(Q^{op})$.
\end{remark}

\begin{example}
We consider the pseudo-torus under diagonal reflection of $\bb{Z}/2$. Define the pseudo-circle $\sigma^1 = \{a, b, c, d\}$ equipped with the order $c \leq a$, $d \leq a$, $c \leq b$, $d \leq b$. It is called the pseudo-circle since $|\scr{K}(\Sigma^1)| \cong S^1$. 
 
 We define the pseudo-torus as the product poset $T_{ps} = \Sigma^1 \times \Sigma^1$ pictured in \ref{fig:ryanspersonalhell}. Define the $\bb{Z}/2$ action as $\mu: (p, q) \mapsto (q, p)$. One can check that this group action is continuous and that the fixed point space is the path-connected space $\{(a,a), (b,b), (c, c), (d, d)\}$. Suppose the maximal point $(a,b)$ is contained in an open invariant set $U \subset T_{ps}$. Then $(b,a) \in U$. Since $U$ contains the pseudo circle $\{(a,b), (c,d), (b,a), (d,c)\}$, we have that $U$ is not homotopic to the orbit of a single point. Then $\LS_{\bb{Z}/2}(T_{ps}) = \infty$.
\end{example}
  \begin{figure}
      \centering
    
\begin{tikzpicture}[scale=1.25]

\node (bd) at (0, 0, -3) {$bd$};
\node (bc) at (0, 0, -5) {$bc$};
\node (ba) at (0, 1, -4) {$ba$};
\node (bb) at (0, -1, -4) {$bb$};

\node (ad) at (0, 0, 3) {$ad$};
\node (ac) at (0, 0, 5) {$ac$};
\node (aa) at (0, 1, 4) {$aa$};
\node (ab) at (0, -1, 4) {$ab$};

\node (cd) at (-3, 0, 0) {$cd$};
\node (cc) at (-5, 0, 0) {$cc$};
\node (ca) at (-4, 1, 0) {$ca$};
\node (cb) at (-4, -1, 0) {$cb$};

\node (dd) at (3, 0, 0) {$dd$};
\node (dc) at (5, 0, 0) {$dc$};
\node (da) at (4, 1, 0) {$da$};
\node (db) at (4, -1, 0) {$db$};

\draw[black, ->] (dc) -- (bc); % Outer circle back
\draw[black, ->] (cc) -- (bc);

\draw[black, ->] (bc) -- (ba); % North circle
\draw[black, ->] (bc) -- (bb);
\draw[white, line width=4pt] (bd) -- (ba);
\draw[black, ->] (bd) -- (ba);
\draw[black, ->] (bd) -- (bb);

\draw[black, ->] (cb) -- (ab); % Bottom circle
\draw[black, ->] (cb) -- (bb);
\draw[black, ->] (db) -- (ab);
\draw[black, ->] (db) -- (bb);

\draw[white, line width=4pt] (cd) -- (ad);
\draw[black, ->] (cd) -- (ad); % Inner circle
\draw[black, ->] (cd) -- (bd);
\draw[white, line width=4pt] (dd) -- (ad);
\draw[black, ->] (dd) -- (ad);
\draw[white, line width=4pt] (dd) -- (bd);
\draw[black, ->] (dd) -- (bd);

\draw[black, ->] (cc) -- (ca); % West Circle
\draw[black, ->] (cc) -- (cb);
\draw[white, line width=4pt] (cd) -- (ca);
\draw[black, ->] (cd) -- (ca);
\draw[black, ->] (cd) -- (cb);

\draw[black, ->] (dc) -- (da); % East Circle
\draw[black, ->] (dc) -- (db);
\draw[white, line width=4pt] (dd) -- (da);
\draw[black, ->] (dd) -- (da);
\draw[black, ->] (dd) -- (db);

\draw[white, line width=4pt] (ca) -- (aa);
\draw[black, ->] (ca) -- (aa); % Top circle
\draw[black, ->] (ca) -- (ba);
\draw[white, line width=4pt] (da) -- (aa);
\draw[black, ->] (da) -- (aa);
\draw[black, ->] (da) -- (ba);

\draw[white, line width=4pt] (ac) -- (aa);
\draw[black, ->] (ac) -- (aa); % South Circle
\draw[black, ->] (ac) -- (ab);
\draw[black, ->] (ad) -- (aa);
\draw[black, ->] (ad) -- (ab);

\draw[white, line width=4pt] (cc) -- (ac);
\draw[black, ->] (cc) -- (ac); % Outer circle front
\draw[white, line width=4pt] (dc) -- (ac);
\draw[black, ->] (dc) -- (ac);
\end{tikzpicture}
      
      \caption{The Pseudo Torus}
      \label{fig:ryanspersonalhell}
  \end{figure}
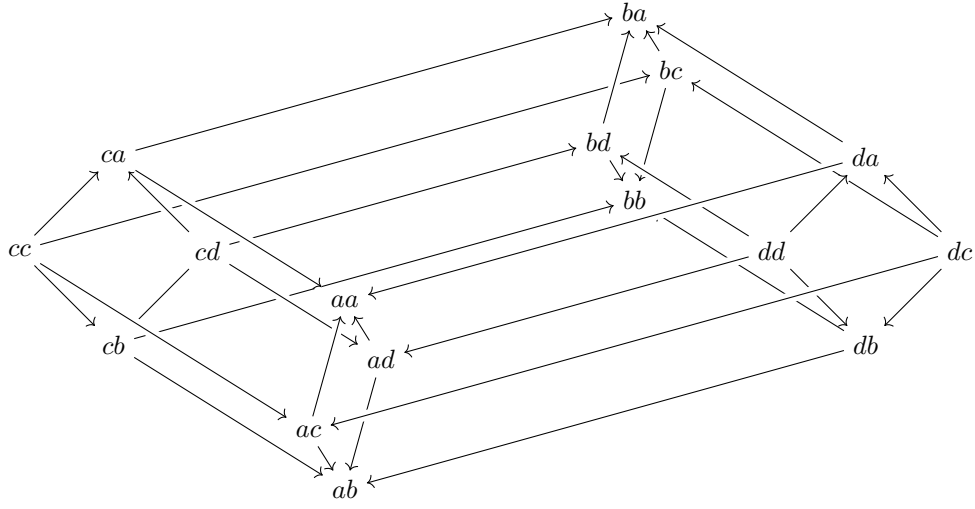

Now, we investigate how subdividing a space affects its topological complexity. Specifically, we prove the surprising result that one need only subdivide a poset twice in order for its equivariant topological complexity to become finite. However, we need to prove some intermediate results about subdivision and group actions first.

\begin{proposition}\label{prop:subdivisionContract}
Let $X$ be a finite $T_0$ $G$-space. Let $P$ be the downwards closure of a point in $\sd^2(X)$. Then $G\cdot P$ is an open, $G$-categorical subset of $\sd^2(X)$.
\end{proposition}

\begin{proof}
    Suppose that $P$ is the downwards closure of $S=\{s_1\subsetneq\cdots\subsetneq s_n\}$, i.e. $P=\mathbb{P}(S)\setminus \varnothing$. 
	
	We define a map 
	\[
	\begin{tikzcd}[row sep=0em,ampersand replacement=\&]
	f_1: \&[-3em] G\cdot P \arrow[r] \& G\cdot P \\
	  \& g\cdot U\arrow[r,mapsto] \& \begin{cases}
	   g\cdot S & s_n\in U\\
	   g\cdot U & \text{else}. 
	  \end{cases}
	\end{tikzcd}
	\]
	To see that $f_1$ is well-defined, we note that if $g\cdot U=h\cdot V$, and $s_n\in U$, then $h^{-1}g\cdot U=V\subseteq S$. Since the only element of $S$ of cardinality $\#(s_n)$ is $s_n$, we see that $h^{-1}g\cdot s_n= s_n\in V$. Similarly, if $g\cdot U=h\cdot V$ and $s_n\in U$, then $g\cdot s_n=h\cdot s_n$, and so $g\cdot S=h\cdot S$. 
	
	We also define a map
	\[
	\begin{tikzcd}[row sep=0em,ampersand replacement=\&]
	f_2: \&[-3em] G\cdot P \arrow[r] \& G\cdot P \\
	\& g\cdot U\arrow[r,mapsto] \& g\cdot (U\setminus \{s_n\}) 
	\end{tikzcd}
	\]
	To see that $f_2$ is well-defined, we note, as above, if $g\cdot U=h\cdot V$, then (1) $s_n\in U$ if and only if $s_n\in V$ and (2) if $s_n\in U$, then $g\cdot s_n=h\cdot s_n$. 
	
	It is immediate from construction that $f_2\leq f_1\geq \on{Id}_{G\cdot P}$. Routine checks confirm that $f_1$ and $f_2$ are monotone and equivariant.
	
	\iffalse
	It thus remains for us to show that $f_1$ and $f_2$ are monotone and equivariant. 
	
	We first examine monotonicity. It is clear that $f_2$ is monotone. To see that $f_1$ is monotone, we need not consider cases $g\cdot U\subseteq h\cdot V$ where either $U$ and $V$ both contain $s_n$, or neither contain $s_n$. Note that if $U\subseteq S$ contains $s_n$ and $V\subseteq S$ does not, then either $g\cdot U$ and $h\cdot V$ are incomparable, or $h\cdot V\subset g\cdot U$. In the latter case, we immediately have that 
	\[
	h\cdot V\subseteq g\cdot U\subseteq g\cdot S. 
	\] 
	Thus, $f_1$ is monotone. 
	
	To check equivariance, we note that, for $g\cdot U\in G\cdot P$ and $h\in G$, we have 
	\[
	h\cdot f_1(g\cdot U)=\begin{cases}
	h\cdot (g\cdot S) & s_n\in U\\
	h \cdot (g\cdot U) & \text{else} 
	\end{cases} =f_1(hg\cdot U)=f_1(h\cdot (g\cdot U))
	\]
	and 
	\[
	h\cdot f_2(g\cdot U)=h\cdot g\cdot(U\setminus\{s_n\})=hg\cdot (U\setminus s_n)=f_2(hg\cdot U)=f_2(h\cdot (g\cdot U))
	\]
	showing the equivariance of both maps.
	
	\fi

	Defining $P_k$ to be the downwards closure of $\{s_1\subsetneq \cdots \subsetneq s_k\}$ in $\sd^2(X)$, we get a filtration 
	\[
	G\cdot P_1\subseteq G\cdot P_2 \subseteq \cdots \subseteq G\cdot P_n=g\cdot P.
	\]
	
	Above, we have defined a combinatorial $G$-homotopy from $\on{Id}_{G\cdot P_n}$ to a map $f_2$ taking values in $G\cdot P_{n-1}$. Iterating this argument shows that $\on{Id}_{G\cdot P}$ is combinatorially $G$-homotopic to a map taking values in $G\cdot P_1=G\cdot s_1=\cal{O}(s_1)$. Thus, $G\cdot P$ is $G$-categorical.
\end{proof}

The proposition leads us to a significant corollary.

\begin{corollary}\label{cor:LSG_finite_suff_small}
Let $X$ be a finite $T_0$ $G$-space. Then $\LS_G(\sd^n(X)) < \infty$ for all $n \ge 2$. 
\end{corollary}

\begin{proof}
Every point of $\sd^2(X)$ lies in the downwards closure $U$ of a point in $\sd^2(X)$. By Proposition~\ref{prop:subdivisionContract} this means every point of $\sd^2(X)$ is contained in an open $G$-categorical subspace.
\end{proof}

%Given the inequality between $\TC_G(X)$ and $\LS_G(X \times X)$, this corollary almost shows that $\TC_G(\sd^2(P))$ is always finite. One small proposition about the fixed points of a subdivision is necessary for the result.

\begin{lemma}\label{lem:sdFixedPts}
Let $X$ be a finite $T_0$ $G$-space. Then for all $n \in \NN$ we have $\sd^n(X^G) = \sd^n(X)^G$.
\end{lemma}

%\begin{proof}
%Suppose $x \in \sd(X^G)$. Then $x \in \sd(X)$ and the vertices of $x$ in $\mathcal K(X^G)$ are in $X^G$. Since the vertices of $x$ are fixed, $x$ is fixed. So $\sd(X^G) \subseteq \sd(X)^G$. 

%Suppose that $y \in \sd(X)^G$. Let $g \in G$ and let $y = \{a_1 \le \dots \le a_n\}$. Then $a_1, \dots, a_n \in X$. Since $y$ is a fixed-point $g \cdot a_i$ is a vertex of $y$ for all $i \in [n]$. Then, since $g$ is monotone $g \cdot a_i = a_i$. Then $a_i \in X^G$. Then $y \in \sd(X^G)$. The claim follows inductively.
%\end{proof}

Now we are able to prove the main result of this section.
% I'm promoting this to a theorem. It's a pretty big result. --Ryan
\begin{theorem}\label{thm:finiteTCG}
Let $X$ is a finite $G$-connected $T_0$ space. Then $\TC_G(\sd^n(X)) < \infty$ for $n \ge 2$.
\end{theorem}

\begin{proof}
Let $H$ be a closed subgroup of $G$. Since $X$ is $G$-connected, $X^H$ is path-connected. Then $\sd^n(X^H)$ is path-connected and $\LS_G(\sd^n(X)) < \infty$ by Corollary~\ref{cor:LSG_finite_suff_small}. Then $\sd^n(X)^H$ is path-connected by Lemma~\ref{lem:sdFixedPts}. Therefore $\sd^n(X)$ is $G$-connected. Since $\sd^n(X)$ is $G$-connected, we have $\TC_G(\sd^n(X)) \le \LS_G(\sd^n(X) \times \sd^n(X)) \le \#(G) \, \LS_G(\sd^n(X))^2$ by Proposition~\ref{prop:LSofProduct}. So we have 
$$
\TC_G(\sd^n(X)) \le \#(G) \, \LS_G(\sd^n(X))^2 < \infty.
$$ as desired.
\end{proof}

\section{Equivariant Combinatorial Complexity}\label{sec:CC_G}

In this section, we generalize Tanaka's notion of {\it combinatorial complexity} from \cite{tanaka_2018} to the equivariant setting. We follow his approach of defining a descending sequence of invariants which eventually converge to the topological complexity of our space. Many of the proofs in this section are nearly identical to their counterparts in \cite{tanaka_2018} once we show that the relevant maps are equivariant. 

\begin{definition}
 For any $m \geq 0$ and any finite $T_0$ $G$-space $X$, define $\CCC_{G, m}(X) := \secat_G(q_m)$ where $q_m: X^{J_m} \to X \times X$ is the combinatorial path fibration. That is, $\CCC_{G, m}(X)$ is the smallest number of open sets we may cover $X$ with that admit a motion planning which maps two points in $X$ to a path of length $m$ or less.
\end{definition}

This definition gives us a sequence of invariants. The following lemma shows that this sequence of invariants is well defined.

\begin{lemma}\label{lem:CCWellDefined}
For any $m \geq 0$ and any finite space $X$, we have $\CCC_{G, m}(X) \geq \CCC_{G, m+1}(X)$.
\end{lemma}

\begin{proof}
The is proof is \textit{mutatis mutandis} the same as that of \cite[Lem 2.3]{tanaka_2018}, one need only check that $r^*$ in loc. cit. is equivariant.

%except that we must check that the map $r^*$ in the following commutative diagram is equivariant.
% https://q.uiver.app/?q=WzAsMyxbMCwwLCJYXntKX219Il0sWzIsMCwiWF57Sl97bSsxfX0iXSxbMSwxLCJYIFxcdGltZXMgWCJdLFswLDEsInJeKiJdLFswLDIsInFfbSIsMl0sWzEsMiwicV97bSsxfSJdXQ==
\[\begin{tikzcd}
	{X^{J_m}} && {X^{J_{m+1}}} \\
	& {X \times X}
	\arrow["{r^*}", from=1-1, to=1-3]
	\arrow["{q_m}"', from=1-1, to=2-2]
	\arrow["{q_{m+1}}", from=1-3, to=2-2]
\end{tikzcd}\]
We see that $r^*$ is equivariant from the equalities
\begin{align*}
    (r^*(g \cdot \alpha))(t) = ((g \cdot \alpha) \circ r)(t) = (g \cdot \alpha)(r(t)) = g\alpha(r(t)) = g(\alpha \circ r)(t) = g(r^*(\alpha))(t)
\end{align*}
completing the proof.
\end{proof}

Now that we have a well-defined, decreasing indexed sequence of invariants, we may define a minimum invariant of the sequence.

\begin{definition}
 Let $X$ be a finite $T_0$ space. Due to lemma~\ref{lem:CCWellDefined}, we may now define the {\it equivariant combinatorial complexity} of $X$ to be
 \[
 \CCC_G(X) := \lim_{m \to \infty} \CCC_{G, m}(X) = \min_{m \geq 1} \CCC_{G, m}(X)
 \]
\end{definition}

In order to show that $\TC_G(P) = \CCC_G(P)$, we give a summary result about the universal property of the compact-open topology as it applies to equivariant maps. Since we want results for both combinatorial paths and topological paths, we use an arbitrary locally compact space $K$ that we may substitute for either $I$ or $J_m$ to retrieve our results about both types of path.

It is also necessary to fix some notation for the following lemma. If $X$ and $Y$ are $G$-spaces and $X$ is locally compact, we denote by $\overline{Y^X}$ the set of continuous, equivariant maps $X \to Y$, and we equip it with the subspace topology inherited from $Y^X$. 

\begin{lemma}\label{lem:GUniversal}
Suppose $X, Y,$ and $K$ are $G$-spaces, and $X$ and $K$ are locally compact, where $G$ acts trivially on $K$ and pointwise on $Y^K$. Then the set of continuous, equivariant maps $X \times K \to Y$ is in one-to-one correspondence with the set of continuous, equivariant maps $X \to Y^K$. These sets are also in one-to-one correspondence to the set of continuous maps $K \to \overline{Y^X}$. The correspondence is such that if $f: X \times K \to Y$, $g: X \to Y^K$ and $h: K \to \overline{Y^X}$ all correspond to each other, then 
\[
    f(x, k) = [g(x)](k) = [h(k)](x)
\]
for each $x \in X$ and $k \in K$. %canonical bijections $\sf{Hom}(X, Y^K) \cong \sf{Hom}(X \times K, Y) \cong \sf{Hom}(K, Y^X)$ , obtained from the universal property of the compact-open topology descend to an equivariant set of otherwise identical bijections.
\end{lemma}

Equipped with this lemma, we can prove our desired result about the equality of Coleman and Grant's equivariant topological complexity and our equivariant combinatorial complexity. 

\begin{theorem} \label{thm:tc_g=cc_g}
For any finite $G$-space $P$,
\[
  \TC_G(P) = \CCC_G(P).
\]

\begin{proof}
The arguments in this proof follow closely to \cite[Thm 3.2]{tanaka_2018} except we verify that the relevant maps are equivariant at each step. First we show that $\CCC_G(P) \geq \TC_G(P)$. Assume that $\TC_G(P) = n$ with an open, invariant cover $\{Q_i\}_{i=1}^n$ of $P \times P$ and a continuous $G$-section $Q_i \to P^I$ for each $i$. We apply \ref{lem:GUniversal} to get a continuous map $I \to \overline{P^{Q_i}}$ (where $\overline{P^{Q_i}}$ denotes the space of continuous, equivariant maps $Q_i \to P$). So we get $J_m \to \overline{P^{Q_i}}$ for some $m \geq 0$. Applying \ref{lem:GUniversal} again, we arrive at our combinatorial $G$-section $Q_i \to P^{J_m}$. Since every $Q_i$ admits a combinatorial $G$-section, we have that $\CCC_G(P) \leq n$. 

Now we show the reverse inequality. Assume that $\CCC_G(P) = n$. Then there exists an open, invariant cover $\{Q_i\}_{i=1}^n$ of $P \times P$ and continuous $G$-section $s_i: Q_i \to P^I$ of $q_m$ for each $i$ and some $m \geq 0$. Let $\alpha_m: [0, m] \to J_m$ be defined as in loc. cit. for $k = 0, 1, \dots$. This map is known to be continuous, and it preserves both ends. Define $\beta: I \to J_m$ by $\beta(t) := \alpha(mt)$. This induces a continuous map $\beta^*: P^{J_m} \to P^I$ given by $\beta^*(\gamma) = \gamma \circ \beta$. In particular, $\beta^*$ is equivariant, since for any $g \in G$ and $t \in I$,
\begin{align*}
g(\beta^*(\gamma))(t) &= g(\gamma \circ \beta)(t) = g\gamma(\beta(t)) = (g \cdot \gamma)(\beta(t))\\
&= ((g \cdot \gamma) \circ \beta)(t) = (\beta^*(g \cdot \gamma))(t)
\end{align*}
From this point the proof follows exactly as in \cite[Thm3.2]{tanaka_2018}.
\end{proof}
\end{theorem}

We provide a bound on the number $n$ such that $\CCC_{G, n}(X) = \CCC_G(X)$. This bound only uses the equivariance of the maps. Perhaps one could achieve a better bound once continuity is taken into account.

\begin{lemma}\label{lem:EquiMapsNumber}
Let $A$ and $B$ be finite {\it sets} and let $G$ act on $A$ and $B$. Then the number of equivariant maps $f: A \to B$ is at most $\#(B)^{\#(A / G)}$.
\end{lemma}

\begin{proof}
 We need only note that each equivariant map is uniquely determined by the
induced map on a chosen set of orbit representatives. The number of such maps
 $A \to B$ is at most $\#(B)^{\#(A / G)}$. 
\end{proof}

%The following proposition bounds the number of iterations one must perform in order to compute the equivariant combinatorial complexity of a $G$-poset.

\begin{proposition}\label{prop:CCpathlength}
Let $X$ be a finite $T_0$ $G$-space and let $N = \#(X)^{\#(X \times X / G)}$. Then $\CCC_G(X) = \CCC_{G, N}(X)$.
\end{proposition}

\begin{proof}
We immediately get that that $\CCC_G(X) \leq \CCC_{G, N}(X)$. Now, we show the opposite inequality. Let $U \subseteq X \times X$ be an open, invariant set admitting an equivariant section $s: U \to X^{J_m}$ of the path fibration $q_m$. We will show that it also admits a section of the path fibration $q_N$.

Given $s: U \to X^{J_m}$, there exists a path $s_2: J_m \to \overline{X^U}$ by Lemma~\ref{lem:GUniversal} and  with the procedure from Proposition~\ref{prop:PathLength}. Using Lemma~\ref{lem:EquiMapsNumber}, we may change the path $s_2$ into a path $s_3: J_N \to \overline{X^U}$, while preserving the endpoints. From here, we may apply Lemma~\ref{lem:GUniversal} again to turn $s_3$ into a map $s_4: U \to X^{J_N}$. The reader may calculate that
\begin{align*}
i_U(u) &= (q_m \circ s_1)(u) = (q_N \circ s_4)(u)
\end{align*}

Since $i_U = q_N \circ s_4$, it follows that $U$ is a $G$-sectional set for $q_N$. Thus, $\CCC_{G, N}(X) \leq \CCC_G(X)$.
\end{proof}

\section{Equivariant Simplicial Complexity}\label{sec:SC_G}

\subsection{Group Actions on Simplicial Complexes} 

Here, we review some results about ordered $G$-simplicial complexes and equivariant simplicial approximations of equivariant simplicial maps. 

\begin{definition}
If $K$ is a simplicial complex with a vertex set $V$ and there is a finite group $G$ that acts on $V$ such that
\begin{itemize}
    \item The map $g \cdot (-): V \to V$ is a map of simplicial complexes for any $g \in G$
    \item If $e \in G$ is the identity, then $e \cdot v = v$ for every $v \in V$
    \item For any $g, h \in G$ and any $v \in V$, the equality $h \cdot (g \cdot v) = (hg) \cdot v$ holds.
\end{itemize}
Then $K$ is a \textit{$G$-Simplicial Complex}. It follows that $|K|$ is a $G$-space as well.
\end{definition}

When we consider simplicial complexes equipped with a simplicial $G$-action ($G$-simplicial complexes), it becomes necessary to add structure to make the fixed-points well-behaved.

%As it happens, the lack of structure on an arbitrary simplicial complex becomes problematic when we add a $G$-action. For example, if we consider the complex $\Delta^1 = \{a, b\}$ we may endow it with a $\ZZ / 2$ action which switches $a$ and $b$. Then $\sd(\Delta^1)$ has a fixed point $\{a, b\}$, but $\Delta^1$ has no fixed points, so there is no equivariant map between them. Indeed, there is not in general an equivariant map $\sd(K) \to K$ for an arbitrary $G$-simplicial complex $K$. Since subdivision is so important in finding simplicial approximations, this spells doom for any attempt to approximate equivariant maps for arbitrary $G$-simplicial complexes. It is clear that we must restrict the $G$-action to make equivariant approximation possible.

\begin{definition}
Given a simplicial complex $K$ equipped with a partial order on its vertex set, such that every simplex is totally ordered and a group $G$ acting on an ordered simplicial complex $K$ such that $g \cdot (-): K \to K$ is both monotone and a map of simplicial complexes for every $g \in G$, then $K$ is called an {\it ordered $G$-simplicial complex}. 
\end{definition}

As we shall see later, it is always possible to approximate equivariant maps into the realization of an ordered $G$-simplicial complex.

\begin{remark}
If $P$ is a finite poset, then $\scr{K}(P)$ is an ordered simplicial complex. If $P$ is a finite $G$-poset, then $\scr{K}(P)$ is an ordered $G$-simplicial complex. Similarly, for any simplicial complex $K$, $\sd^n(K)$ is an ordered simplicial complex.
\end{remark} 

Now that we have added the necessary structure to our simplicial complexes, we must define a notion of equivariance on maps between them.

\begin{definition}
We call a monotone simplicial map $\phi: K \rightarrow L$ between ordered $G$-simplicial complexes \emph{equivariant} if, for any $\sigma \in \Sim(K)$ and $g \in G$, we have $\phi(g \cdot \sigma) = g \cdot \phi(\sigma)$. The realization of such a map is an equivariant map of $G$-spaces.
\end{definition}

As in the non-equivariant setting, it is possible to approximate homotopies using the combinatorics of simplicial maps.

\begin{definition}
Given two equivariant simplicial maps $\phi, \psi :K \to L$, say that they are \textit{equivariantly $c$-contiguous} if they are $c$-contiguous via a chain
\[
\phi = \phi_0, \phi_1, \dots, \phi_c = \psi
\]
where each $\phi_i$ is equivariant.
\end{definition}

We also introduce an equivariant form of simplicial approximations.

\begin{definition} 
For two ordered $G$-simplicial complexes, $K$ and $L$, and an equivariant map $f: |K| \to |L|$, an \textit{equivariant simplicial approximation} of $f$ is an equivariant monotone simplicial map $\phi: K \to L$ which is also a simplicial approximation of $f$.
\end{definition}

We focus on the case of ordered simplicial complexes, as this is the setting where simplicial approximation works easily.

\begin{example}\label{ex:Tidapprox}
Recall the equivariant map $\tau: \sd(X) \to X$ from Lemma~\ref{lem:EquivarT}. If $X$ is a simplicial complex, then $\tau$ is a simplicial map. For an ordered $G$-simplicial complex, $\tau$ is equivariant. For any simplicial complex $K$, it is easy to check that $\tau: \sd(K) \to K$ is also an equivariant approximation of the identity. 
\end{example}
By the above example, there is always a simplicial approximation of the identity from $\sd(K)$ to $K$. This fact will be important later on. Now we explore the properties of equivariant contiguity.

\begin{proposition}\label{prop:ContiguitySummary}
The following statements about equivariant contiguity hold.
\begin{enumerate}
    \item Two equivariant simplicial approximations to the same continuous map are  equivariantly $1$-contiguous.
    \item Let $K$ be a $G$-simplicial complex, and $L$ an ordered $G$-simplicial complex. 
    \item Let $f,p: |K|\to |L|$ be equivariant continuous maps which are equivariantly homotopic. Then there is an $n>0$ and simplicial approximations $\phi,\psi:\sd^n(K)\to L$ of $f$ and $p$, respectively, such that $\phi$ is equivariantly contiguous to $\psi$. 
\end{enumerate}
\end{proposition}

\begin{corollary}\label{cor:GContiguity}
Let $K$ be a $G$-simplicial complex, and $L$ an ordered $G$-simplicial complex. 
Let $f,p: |K|\to |L|$ be equivariant continuous maps which are equivariantly homotopic. Then there is an $n_0>0$ such that, for each $n\ge n_0$, any pair of simplicial approximations $\phi,\psi:\sd^n(K)\to L$ of $f$ and $p$, respectively, are equivariantly contiguous. 
\end{corollary}

Now we recall some useful constructions on the realization of a simplicial complex.

\begin{definition}
Let $K$ be a simplicial complex and $S \in \Sim(K)$ be an $n$-simplex of $K$. The \emph{open simplex} $S^{\circ} \in |K|$ is defined by:
$$
S^{\circ} := \left \{ x \in \RR^{n + 1} \mid \sum_{i = 1}^{n + 1} x_i = 1 \text{ and } 0< x_i < 1 \, \forall i = 1, \dots, n+1 \right \}.
$$
\end{definition}

\begin{definition}
Let $K$ be a simplicial complex and let $x \in K$. We define the \emph{star of x} to be 
$$
\star(x) := \bigcup_{S \in \Sim(K) \mid x \in S} S^{\circ} \subseteq |K|.
$$
\end{definition}

We now establish an analogue of Proposition~\ref{prop:subdivisionContract} in the simplicial setting. As before, this allows us to bound the topological complexity from above.

\begin{lemma}\label{lem:starOpen}
Let $K$ be a simplicial complex and let $x \in K$. Then $\star(x)$ is an open subset of $|K|$.
\end{lemma}

A common concern is understanding how open sets of geometric realizations of simplicial complexes are acted upon by the group action. By restricting our examination to geometric realizations of order complexes of finite $G$-posets we greatly restrict the behaviour of the group actions. In particular, an element of an open simplex $\sigma^\circ$ is sent to itself by a group element if and only if each element of the realization $|\sigma|$ of the entire simplex is sent to itself the group element.

\begin{proposition}\label{prop:starGCategorical}
Let $P$ be a finite $G$-poset and $K = \scr{K}(P)$ be the resulting ordered $G$-simplicial complex. Let $x \in K$. Then $G \cdot \star(x)$ is $G$-categorical.
\end{proposition}

\begin{proof}
Since $G$ acts monotonically on $P$ and since $K = \scr{K}(P)$, we have that $g \cdot (S^{\circ} \cup \{x\})$ is convex for all $S \in \Sim(K)$ such that $x \in S$ and for all $g \in G$. Since $P$ is a $G$-poset, if $y \in \star(x)$ then $g \cdot y \in g \cdot \star(x)$ for all $g \in G$. Then the following is a continuous function in $|K|$: 
$$
H : G \cdot \star(x) \times I \rightarrow |K|
$$
defined by $H(g \cdot y,t) = (g \cdot x)t + (g \cdot y)(1 - t)$. Since $G$ acts monotonically on the ordered poset $P$ and $x \in |\sigma|$ for some $\sigma \in \Sim(K)$, we have that if $g,h \in G$, $y,z \in \star(x)$, and $g \cdot y = h \cdot z$, then $g \cdot x = h \cdot x$. Therefore, if $g,h \in G$, $y,z \in \star(x)$, and $g \cdot y = h \cdot z$, then we have that $H(g \cdot y,t) = H(h \cdot z, t)$ for all $t \in I$. 
%So $H(g \cdot y,t) = (g \cdot x)t + (g \cdot y)(1 - t)$ is independent of choice of $g \in G$.

Since $\star(x)$ is an open subset of $|K|$, the restriction $H|_{G \cdot \star(x) \times \{0\}} = \iota_{G \cdot \star(x)}$ and the restriction $H|_{G \cdot \star(x) \times \{1\}}$ is a continuous function that maps $G \cdot \star(x)$ to $\cal O_x$.
\end{proof}

\subsection{G-Connectedness of Simplicial Complexes}

In this subsection we establish that the realization of an order complex of a $G$-poset $P$, denoted $|\scr K(P)|$, is $G$-connected if and only if the $G$-poset $P$ is $G$-connected. As a corollary to the previously stated result, we find that if $P$ is $G$-connected, then $\TC_G(|\scr K(P)|)$ is finite.

\begin{lemma}\label{lem:RealizationFixedPts}
Let $K$ be an ordered $G$-simplicial complex. Then $|K|^G = |K^G|$. 
\end{lemma}

\iffalse

\begin{proof}
Suppose $x \in |K^G|$ and let $V_{K^G}$ be the vertex set of $K^G$. Then $x = \lambda_1v_1 + \dots + \lambda_nv_n$ for $v_1, \dots, v_n \in V_{K^G}$. Since $K^G$ is the $G$-fixed point set of $K$, then $g \cdot v_i = v_i$ for any $v_i$. It follows that
\begin{align*}
    g \cdot x &= g \cdot (\lambda_1v_1 + \dots + \lambda_nv_n) = \lambda_1(g \cdot v_1) + \dots + \lambda_n(g \cdot v_n)\\
    &= \lambda_1v_1 + \dots + \lambda_nv_n = x
\end{align*}
Thus, $x$ is a $G$-fixed point of $|K|$, and so $x \in |K|^G$. Thus, $|K^G| \subseteq |K|^G$.

Next, let $x \in |K|^G$, and let $\sigma \in \Sim(K)$ be the unique simplex where $x \in \sigma^\circ$. Since $K$ is an ordered simplex, every $g \in G$ which fixes $x$ also fixes the vertices of $\sigma$. Thus, $\sigma \subseteq K^G$, and so $x \in |\sigma| \subseteq |K^G|$. Hence, $|K|^G \subseteq |K^G|$, and the result follows.
\end{proof}

The above lemma proves that the fixed points of a space of somewhat set. Next we show the relationship between a poset and its order complexes as spaces.

\fi

\begin{lemma}\label{lem:PathInRealization}
Suppose $P$ is a finite order-connected poset. Then there is a path in $|\scr{K}(P)|$ between any two vertices of $\scr{K}(P)$. 
\end{lemma}

\iffalse

\begin{proof}
The vertices of $\scr{K}(P)$ are exactly the points in $P$. Choose two vertices $x, y \in P$. Then there is a combinatorial path from $x$ to $y$. Since we may reverse and concatenate topological paths, it suffices to show that if $x \leq y$, then there is a topological path from $x$ to $y$ in $|\scr{K}(P)|$. If $x \leq y$, then $\{x, y\} \in \Sim(\scr{K}(P))$ is a $1$-simplex. Then $|\{x, y\}| \cong |\Delta^1|$ is a path-connected space containing $x$ and $y$, so there is a path between them. Since $|\{x, y\}| \subseteq |\scr{K}(P)|$, this path is also a path in $|\scr{K}(P)|$. The result follows. 
\end{proof}

\fi

Now that we have our lemmas we introduce a useful proposition.

\begin{proposition}\label{prop:GConnectRight}
Let $P$ be a finite $G$-poset and $K = \scr{K}(P)$ be the resulting ordered $G$-simplicial complex. If $P$ is $G$-connected, then $|K|$ is $G$-connected.

\begin{proof}
Suppose $P$ is $G$-connected and let $H \subseteq G$ be a closed subgroup. Take $x, y \in |K|^H$. Since by Lemma~\ref{lem:RealizationFixedPts} $|K|^H = |K^H|$, we may take simplices $\sigma_x, \sigma_y \in \Sim(K^H)$ such that $x \in |\sigma_x|$ and $y \in |\sigma_y|$. Choose vertices $v_x, v_y$ of $\sigma_x$ and $\sigma_y$ respectively. Since $P$ is $G$-connected, then $P^H$ is order connected, and so by Lemma~\ref{lem:PathInRealization} there is a path $\beta: I \to |K^H|$ from $v_x$ to $v_y$. Since $\sigma_x$ is a simplex, then $|\sigma_x|$ is path-connected, and so there is a path $\alpha: I \to |K^H|$ from $x$ to $v_x$. Similarly, since $\sigma_y$ is a simplex, then $|\sigma_y|$ is path-connected, and so there is a path $\gamma: I \to |K^H|$ from $v_y$ to $v_y$. Thus, we achieve a path $\alpha * \beta * \gamma$ from $x$ to $y$. Since we may repeat this process for any $x, y \in |K|^H$, then $|K|^H$ is path-connected. Since this holds for any closed subgroup $H \subseteq G$, then $|K|$ is $G$-connected.  
\end{proof}
\end{proposition}

Now we show the other direction.

\begin{proposition}\label{prop:GConnectLeft}
Let $P$ be a finite $G$-poset and $K = \scr{K}(P)$ be the resulting ordered $G$-simplicial complex. Then if $|K|$ is $G$-connected then $P$ is $G$-connected.
\end{proposition}

\begin{proof}
Let $H$ be a closed subgroup of $G$. Suppose that $|K|$ is $G$-connected. Then $|K|^H=|K^H|$ is path-connected. Then for any $x,y \in K^H$ there is a path
\[
\begin{tikzcd}[ampersand replacement=\&]
\alpha: \&[-3em] {|\scr{K}(J_1)|}\cong I \arrow[r] \& {|K^H|}
\end{tikzcd}
\]
from $x$ to $y$. There is a simplicial approximation $\beta:\sd^n(\scr{K}(J_1))\cong \scr{K}(J_n)\to K^H$ to $\alpha$. Since this is a simplicial approximation, it must send $0$ to $x$ and $n$ to $y$. Consequently, $\beta$ defines a sequence of 1-simplices in $K^H$ connecting $x$ to $y$. Thus, $P^H$ is order-connected.
\end{proof}

Using the bounds on $\TC_G$ implied by $G$-connectedness and the characterization of $G$-categorical subsets of $|K|$ from Proposition~\ref{prop:starGCategorical} we can provide conditions sufficient to ensure that $\TC_G(|K|)$ is finite.

\begin{proposition}\label{prop:finiteTCRealize}
Let $P$ be a finite $G$-connected poset. Then $\TC_G(|\scr K(P)|) < \infty$.
\end{proposition}

\begin{proof}
By Proposition~\ref{prop:GConnectLeft}, since $P$ is $G$-connected $|\scr K(P)|$ is $G$-connected. By Proposition~\ref{prop:GConnectedBound}, since $|\scr K(P)|$ is $G$-connected the following inequality holds:
$$
\TC_G(|\scr K(P)|) \le \LS_G(|\scr K(P)| \times |\scr K(P)|).
$$
By Proposition~\ref{prop:LSofProduct} we have 
$$
\LS_G(|\scr K(P)| \times |\scr K(P)|) \le \#(G)\LS_G(|\scr K(P)|)^2.
$$
By Lemma~\ref{lem:starOpen}, $|\scr K(P)|$ has an open cover $\{G \cdot \star(x) \mid x \in P\}$. By Proposition~\ref{prop:starGCategorical}, for each $x \in P$ $G \cdot \star(x)$ is $G$-categorical. Then $\LS_G(|\scr K(P)|) < \infty$. 
\end{proof}

For any $G$-complex $K$, $\sd(K)$ is the order complex of a poset. We can thus apply the proposition above to more general $G$-complexes.

\subsection{Equivariant Simplicial Complexity} 

With all the relevant machinery in place, we are prepared to generalize Gonz\`{a}lez's notion of simplicial complexity introduced in \cite{gonzalez2018} to the equivariant setting. We closely follow the approach in \cite{gonzalez2018}, defining a sequence of numbers $\SC_G^{b,c}(K)$ which depends on a pair of integers $b$ and $c$ and decreases in each of $b$ and $c$.

\begin{remark}
We may reformulate the definition of equivariant topological complexity in terms of $G$-homotopies. Let $X$ be any $G$-space and denote by $p_1, p_2: X \times X \to X$ the projections. The equivariant topological complexity of $X$ is equivalently, the minimum $n \geq 1$ such that there exists an open, invariant cover $\{U_i\}_{i=1}^n$ of $X \times X$ on which the compositions 
% https://q.uiver.app/?q=WzAsMyxbMCwwLCJVIl0sWzEsMCwiWCBcXHRpbWVzIFgiXSxbMiwwLCJLIl0sWzAsMSwiIiwwLHsic3R5bGUiOnsidGFpbCI6eyJuYW1lIjoiaG9vayIsInNpZGUiOiJ0b3AifX19XSxbMSwyLCJwXzIiLDIseyJvZmZzZXQiOjF9XSxbMSwyLCJwXzEiLDAseyJvZmZzZXQiOi0xfV1d
\[\begin{tikzcd}[ampersand replacement=\&]
	U_i \& {X \times X} \& X
	\arrow[hook, from=1-1, to=1-2]
	\arrow["{p_2}"', shift right=1, from=1-2, to=1-3]
	\arrow["{p_1}", shift left=1, from=1-2, to=1-3]
\end{tikzcd}\]
are $G$-homotopic.
\end{remark}

Since contiguity can be thought of as the simplicial analogue of homotopy, this reformulation motivates the following definition.

\begin{definition}
Let $K$ be a finite $G$-simplicial complex and for any $\ell \geq 0$, fix some equivariant approximations of the identity $\iota: \sd^{\ell+1}(K \times K) \to \sd^\ell(K \times K)$. Define $\SC_G^{b, c}(K)$ as the smallest nonnegative integer $n$ such that $\sd^b(K \times K)$ may be covered by $n$ invariant subcomplexes $\{L_i\}_{i=1}^n$ and the two compositions

% https://q.uiver.app/?q=WzAsNCxbMCwwLCJMX2kiXSxbMSwwLCJcXHNkXmIoSyBcXHRpbWVzIEspIl0sWzIsMCwiSyBcXHRpbWVzIEsiXSxbMywwLCJLIl0sWzAsMSwiIiwwLHsic3R5bGUiOnsidGFpbCI6eyJuYW1lIjoiaG9vayIsInNpZGUiOiJ0b3AifX19XSxbMSwyLCJcXGlvdGEiXSxbMiwzLCJwXzIiLDIseyJvZmZzZXQiOjF9XSxbMiwzLCJwXzEiLDAseyJvZmZzZXQiOi0xfV1d
\[\begin{tikzcd}[ampersand replacement=\&]
	{\pi_1, \pi_2: L_i} \& {\sd^b(K \times K)} \& {K \times K} \& K
	\arrow[hook, from=1-1, to=1-2]
	\arrow["\iota", from=1-2, to=1-3]
	\arrow["{p_2}"', shift right=1, from=1-3, to=1-4]
	\arrow["{p_1}", shift left=1, from=1-3, to=1-4]
\end{tikzcd}\]
are equivariantly $c$-contiguous.
\end{definition}

\begin{remark}
We can always find an equivariant approximation of the identity from $\sd(K) \to K$. If $K$ is a $G$-simplicial complex, the map $\sd(K) \to K$ given by $\{v_1, \dots, v_n\} \mapsto v_n$ is equivariant by Lemma~\ref{lem:EquivarT} and is also known to be an approximation of the identity from Example~\ref{ex:Tidapprox}. By taking $\sd^n(K)$ as our $G$-simplicial complex, we may always find an equivariant approximation of the identity $\sd^{n+1}(K) \to \sd^n(K)$ for any $n \geq 0$.  
\end{remark}

\begin{proposition}
Let $K$ be a finite $G$-simplicial complex. For all $c \ge 0$ we have that $\SC_G^{b, c}(K) \ge \SC_G^{b, c+1}(K) \ge 0$.
\end{proposition}

\begin{proof}
Suppose that $\SC_G^{b, c}(K) =  n$. Then for each $i \in [n]$ the two restrictions $\pi_1, \pi_2: L_i \to K$ are equivariantly $c$-contiguous. Since $\pi_1$ is a simplicial map we have that $\pi_1$ is equivariantly contiguous to itself. So $\pi_1, \pi_2$ are equivariantly $(c + 1)$-contiguous.
\end{proof}

Now that we've proven $\SC_G^{b, c}(K)$ decreases as $c$ approaches infinity, we can define simplicial complexity with respect to only one index.

\begin{definition}
Let $K$ be a finite $G$-simplicial complex. Define 
$$
\SC_G^{b}(K) = \min_{c \geq 0} \SC_G^{b,c}(K)= \lim_{c \rightarrow \infty} \SC_G^{b,c}(K).
$$
\end{definition}

One problem we run into is that the value of $\SC_G^{(b, c)}(K)$ seems to depend on our choice of approximations $\iota: \sd^{m+1}(K) \to \sd^m(K)$. However, this problem resolves itself when taking the limit as $c$ goes to infinity.

\begin{lemma}
The value of $\SC_G^{b}(K)$ is independent of choice of approximations of the identity. 
\end{lemma}

\begin{proof}
The proof is mutatis mutanda that of the proof given in \cite{gonzalez2018}. We take additional care to make sure the contiguity chain is a chain of equivariant functions.
\end{proof}

%\begin{proof}
%Suppose $\SC_G^{b,c}(K) = n$. Fix the equivariant approximations $\iota: \sd^b(K \times K) \rightarrow \sd^{b-1}(K \times K)$ of the identity. Let $\overline{\iota}: \sd^b(K \times K) \rightarrow \sd^{b-1}(K \times K)$ be equivariant approximations of the identity. Two approximations of a continuous map are contiguous. Then the compositions 
%$$
%\pi_i, \overline{\pi_i}: \sd^b(K \times K) \rightarrow K 
%$$
%and $\pi_i \circ j, \overline{\pi_i} \circ j$ where $i = 1,2$ and $j: J \hookrightarrow \sd^b(K \times K)$ are equivariantly contiguous. So $\overline{\pi_1} \circ j, \pi_1 \circ j, \dots, \pi_2 \circ j, \overline{\pi_2} \circ j$ is a contiguity chain of length-$(c+2)$ of equivariant functions. 
%\end{proof}

Now that $\SC_G^b(K)$ is well-defined, we prepare to take the limit as $b \to \infty$. We first prove that this limit converges, and then lastly prove that it converges to $\TC_G(|K|)$.

\begin{lemma}
Let $K$ be a finite $G$-simplicial complex. Then for any $b \geq 1$, $\SC_G^b(K) \geq \SC_G^{b+1}(K)$.
\end{lemma}

\begin{proof}
Let $c \geq 0$ be arbitrary and fix some $\iota: \sd^{m+1}(K \times K) \to \sd^m(K \times K)$ approximations of the identity. Let $J \subseteq \sd^b(K \times K)$ be a subcomplex such that the two restrictions $\pi_1 \circ i_J, \pi_2 \circ i_J: J \to K$ are equivariantly $c$-contiguous. Then there exists an equivariant contiguity chain $\pi_1 \circ i_J = \phi_0, \phi_1, \dots, \phi_c = \pi_2 \circ i_J$ connecting them.

Let $\lambda: \sd(J) \to J$ be an equivariant approximation of the identity. Then the maps $i_J \circ \lambda, \iota \circ i_{\sd(J)}: \sd(J) \to \sd^b(K \times K)$ approximate the inclusion $i_{|J|}: |J| \to |K \times K|$. Since $i_J$ is an approximation of $i_{|J|}$ and since $\lambda$ is an approximation of the identity, we have that $i_J \circ \lambda$ is an approximation of $i_{|J|}$. Similarly, $ \iota \circ i_{\sd(J)}$ is also an approximation of $i_{|J|}$. Therefore, $i_J \circ \lambda$ and $ \iota \circ i_{\sd(J)}$ are $1$-contiguous. Therefore, $\pi_1 \circ i_J \circ \lambda$ and $\pi_1 \circ \iota \circ i_{\sd(J)}$ are $1$-contiguous, and $\pi_2 \circ i_J \circ \lambda$ and $\pi_2 \circ \iota \circ i_{\sd(J)}$ are $1$-contiguous. Thus, we may form an equivariant contiguity chain
\[
\pi_1 \circ \iota \circ i_{\sd(J)}, \pi_1 \circ i_J \circ \lambda = \phi_0 \circ \lambda, \phi_1 \circ \lambda, \dots, \phi_{c-1} \circ \lambda, \phi_c \circ \lambda = \pi_2 \circ i_J \circ \lambda, \pi_2 \circ \iota \circ i_{\sd(J)}
\]
of length $c+2$. However, $\pi_1 \circ \iota \circ i_{\sd(J)}$ is exactly the restriction $\pi_1: \sd(J) \to K$. Similarly, $\pi_2 \circ \iota \circ i_{\sd(J)}$ is exactly the restriction $\pi_2: \sd(J) \to K$. So $\pi_1$ is equivariantly $c$-contiguous to $\pi_2$ on $\sd(J)$. 

Suppose that $\SC_G^{(b, c)}(K) = n$. Then there exists an equivariant cover $\{L_i\}_{i=1}^n$ of $\sd^b(K \times K)$ such that the two restrictions $\pi_1, \pi_2: L_i \to K$ are $c$-contiguous on each $L_i$. Then there exists an equivariant cover $\{\sd(L_i)\}_{i=1}^n$ of $\sd^{b+1}(K \times K)$ such that the two restrictions $\pi_1, \pi_2: L_i \to K$ are $(c+2)$-contiguous. Therefore, $\SC_G^{(b, c)}(K) \geq \SC_G^{(b+1, c+2)}(K)$. Taking the limit as $c$ goes to infinity gives us $\SC_G^b(K) \geq \SC_G^{b+1}(K)$.
\end{proof}

This lemma now makes the definition of equivariant simplicial complexity possible.

\begin{definition}
For any ordered $G$-simplicial complex $K$, define the {\it equivariant simplicial complexity} of $K$ as
\[
\SC_G(K) = \min_{b \geq 0} \SC_G^b(K) = \lim_{b \to \infty} \SC_G^b(K)
\]
\end{definition}

In order to show that $\SC_G(K) = \TC_G(|K|)$ for any ordered $G$-simplicial complex $K$, we need some additional machinery relating to realizations of $G$-simplicial complexes. 

\begin{definition}[Definition 3.11 in \cite{colman_grant_2012}]
Suppose $X$ is a $G$-space. We say $X$ is {\it $G$-completely normal} if whenever $A, B \subseteq X$ are invariant subsets such that $\overline{A} \cap B = \varnothing = A \cap \overline{B}$, then $A$ and $B$ have disjoint, invariant open neighborhoods. 
\end{definition}

Next we establish that topological complexity may be calculated with closed sets as well as open ones.

\begin{lemma}\label{lem:ClosedCover}
Suppose $K$ is an ordered $G$-simplicial complex. Let $C_G(|K|)$ denote the least number of closed, equivariant sets $\{C_i\}$ which cover $|K| \times |K|$ on which the two compositions
\[\begin{tikzcd}[ampersand replacement=\&]
	{C_i} \& {|K| \times |K|} \& {|K|}
	\arrow[hook, from=1-1, to=1-2]
	\arrow["{p_2}"', shift right=1, from=1-2, to=1-3]
	\arrow["{p_1}", shift left=1, from=1-2, to=1-3]
\end{tikzcd}\]
are $G$-homotopic. Then $\TC_G(|K|) = C_G(|K|)$.  
\end{lemma}

\begin{proof}
First we show that $\TC_G(|K|) \geq C_G(|K|)$. Let $\TC_G(|K|) = n$. Then there is an open, equivariant cover $\{U_i\}_{i=1}^n$ such that the compositions
\[\begin{tikzcd}[ampersand replacement=\&]
	{U_i} \& {|K| \times |K|} \& {|K|}
	\arrow[hook, from=1-1, to=1-2]
	\arrow["{p_2}"', shift right=1, from=1-2, to=1-3]
	\arrow["{p_1}", shift left=1, from=1-2, to=1-3]
\end{tikzcd}\]
are  $G$-homotopic. If for every $U_i$ we found a corresponding open, equivariant set $V_i$ such that $\{U_1, \dots, U_{i-1}, V_i, U_{i+1}, \dots, U_n\}$ is still a cover and $\overline{V_i} \subseteq U_i$, then we would have that $\{V_i\}_{i=1}^n$ is an open, equivariant cover of $|K| \times |K|$, and therefore that $\{\overline{V_i}\}_{i=1}^n$ is a closed, equivariant cover of $|K| \times |K|$. Furthermore, we would have that the restrictions of the projections $p_1, p_2: \overline{V_i} \to |K|$ are $G$-homotopic, since they are so on $U_i$, and $\overline{V_i} \subseteq U_i$. Thus, in order to show that $\TC_G(|K|) \geq C_G(|K|)$, it would suffice to show that for any $U_i$ we may find an invariant open $V_i \subseteq |K| \times |K|$ such that $\{U_1, \dots, U_{i-1}, V_i, U_{i+1}, \dots, U_n\}$ is a cover of $|K| \times |K|$ and $\overline{V_i} \subseteq U_i$. 

Take any $U_i$ in the cover. Define $Y_i = \left(\bigcup_{\substack{j=1 \\ j \neq i}}^n U_i\right)^c$, a closed, invariant subset of $U_i$. Then $Y_i$ and $U_i^c$ are disjoint closed, equivariant subsets of $|K| \times |K|$. Since $|K| \times |K|$ is a realization of a $G$-simplicial complex, it is a completely normal $G$-space. By \cite[Lemma 3.12]{colman_grant_2012}, $|K| \times |K|$ is $G$-completely normal, so we may find open, invariant $V_i, W_i \subseteq |K| \times |K|$ such that $Y_i \subseteq V_i$, $U_i^c \subseteq W_i$, and $V_i \cap W_i = \varnothing$. Then, $\overline{V_i}$ is disjoint from $U_i^c$, and so $\overline{V_i} \subseteq U_i$. We also have by definition that $\{U_1, \dots, U_{i-1}, V_i, U_{i+1}, \dots, U_n\}$ is a cover. It follows that $\TC_G(|K|) \geq C_G(|K|)$.

Next, we show that $C_G(|K|) \geq \TC_G(|K|)$. Let $C_G(|K|) = n$. Then there is a closed, equivariant cover $\{C_i\}_{i=1}^n$ of $|K| \times |K|$ such that the composites 
\[\begin{tikzcd}[ampersand replacement=\&]
	{C_i} \& {|K| \times |K|} \& {|K|}
	\arrow[hook, from=1-1, to=1-2]
	\arrow["{p_2}"', shift right=1, from=1-2, to=1-3]
	\arrow["{p_1}", shift left=1, from=1-2, to=1-3]
\end{tikzcd}\]
are $G$-homotopic. Our aim is to find a collection of open, invariant sets $\{U_i\}_{i=1}^n$ such that $C_i \subseteq U_i$ and the composites
\[\begin{tikzcd}[ampersand replacement=\&]
	{U_i} \& {|K| \times |K|} \& {|K|}
	\arrow[hook, from=1-1, to=1-2]
	\arrow["{p_2}"', shift right=1, from=1-2, to=1-3]
	\arrow["{p_1}", shift left=1, from=1-2, to=1-3]
\end{tikzcd}\]
are $G$-homotopic. This suffices to show that $C_G(|K|) \geq \TC_G(|K|)$.

Take some $C_i$ and let $H: C_i \times I \to |K|$ be the homotopy from $p_1$ to $p_2$. By \cite[Theorem 9.3]{Murayama1983}, since $|K| \times |K|$ is a $G$-ANR, we may extend $H$ into a homotopy $\Tilde{H}: (|K| \times |K|) \times I \to X$ such that $\Tilde{H}|_{C \times I} = H$ and $\Tilde{H}|_{(|K| \times |K|) \times \{0\}} = p_1$. We may also apply \cite[Prop. 2.4]{WongPeter1991} to find an open $U_i \subseteq |K| \times |K|$ such that $C_i \subseteq U_i$ and there is a $G$-homotopy $F: U_i \times I \to |K| \times |K|$ from $\Tilde{H}|_{W \times \{1\}}$ to $p_2|_{U_i}$. We may restrict $\Tilde{H}$ and concatenate it with $F$ to get a $G$-homotopy from $p_1|_{U_i}$ to $p_2|_{U_i}$. The two projections on $U_i$ are $G$-homotopic, and so the result follows. 
\end{proof}

\begin{proposition}\label{prop:TCatmostSC}
Suppose that $K$ is any ordered $G$-simplicial complex. Then $\TC_G(|K|) \leq \SC_G(K)$.
\end{proposition}

\begin{proof}
 It suffices to show that $\SC_G^{(b, c)}(K) \geq \TC_G(|K|)$ for any $b, c \geq 0$. Suppose $\SC_G^{(b, c)}(K) = n$ and take an invariant cover $\{L_i\}_{i=1}^n$ of $\sd^b(K \times K)$ on which the restrictions
 \[\begin{tikzcd}[ampersand replacement=\&]
	{\pi_1, \pi_2: L_i} \& {\sd^b(K \times K)} \& {K \times K} \& K
	\arrow[hook, from=1-1, to=1-2]
	\arrow["\iota", from=1-2, to=1-3]
	\arrow["{p_2}"', shift right=1, from=1-3, to=1-4]
	\arrow["{p_1}", shift left=1, from=1-3, to=1-4]
\end{tikzcd}\]
are equivariantly $c$-contiguous. Then $\{|L_i|\}_{i=1}^n$ forms a closed, equivariant cover of $|K| \times |K|$. Furthermore, $|\pi_1|$ and $|\pi_2|$ are $G$-homotopic, since $\pi_1$ and $\pi_2$ are equivariantly contiguous. Since $\pi_1: L_i \to K$ approximates $p_1: |L_i| \to |K|$, and $\pi_2: L_i \to K$ approximates $p_2: |L_i| \to |K|$, we have a chain of $G$-homotopies
\[
p_1 \simeq_G |\pi_1| \simeq_G |\pi_2| \simeq_G p_2
\]
Thus, the two projections are $G$-homotopic on each $|L_i|$. Since $\{|L_i|\}_{i=1}^n$ is a closed, equivariant cover of $|K| \times |K|$ on which the projections are $G$-homotopic, Lemma~\ref{lem:ClosedCover} gives us that $\TC_G(|K|) \leq n$. 
\end{proof}

Now that we have introduced this additional means of finding topological complexity, we can relate simplicial and topological complexity in simplicial complexes.

%\begin{proposition}\label{prop:SCatmostTC}
%Suppose $K$ is any ordered $G$-simplicial complex. Then $\TC_G(|K|) \geq \SC_G(K)$.
%\end{proposition}

%\begin{proof}
%Suppose that $\TC_G(|K|) = n$. Then $|K| \times |K|$ has an open cover $\{U_i\}_{i = 1}^n$ such that each $U_i$ is invariant and has an equivariant continuous motion planning $s_i$. Let $b \ge 0$ be large enough so that each simplex of $\sd^b(K \times K)$ is contained in some $U_i$. Then let $L_i$ be the subcomplex of simplices $\sigma$ whose realizations are contained in $U_i$. Since the realization of each simplex in $L_i$ is contained in $U_i$ we have that, for each $g \in G$, $g\cdot |\sigma| \subseteq U_i$. So $|g \cdot \sigma| \subseteq U_i$. Then $g \cdot \sigma \in L_i$. So $L_i$ is invariant.

%%By Lemma 1.1 in the Gonz\`{a}lez paper
%By \cite[Lemma 1.1]{gonzalez2018}, since $U_i$ admits a continuous motion planning, the two projections $\pi_1, \pi_2: |K| \times |K| \rightarrow |K|$ are homotopic on $U_i$. So $\pi_1, \pi_2$ are homotopic on $L_i$. Then by Corollary \ref{cor:GContiguity} there exists some $b_0 \ge 0$ such that for each $b_1 \ge b_0$ any pair of approximations $\phi_1,\phi_2: \sd^{b_1 + b}(|K| \times |K|) \rightarrow |K|$ of $\pi_1,\pi_2$ respectively are equivariantly $c$-contiguous for some $c \ge 0$. 
%\end{proof}

\begin{theorem}\label{thm:SCeqTC}
Suppose $K$ is any ordered $G$-simplicial complex. Then 
$$
\TC_G(|K|) = \SC_G(K).
$$
\end{theorem}

\begin{proof}
Given Proposition~\ref{prop:TCatmostSC}, it suffices to show that $\TC_G(|K|) \geq \SC_G(K)$. The remainder of the proof is \emph{mutatis mutanda} \cite[Theorem 3.5]{gonzalez2018}. To translate \cite[Theorem 3.5]{gonzalez2018} to the equivariant case, we refer to  Corollary~\ref{cor:GContiguity} to provide a sufficiently large integer $b_0$ such that for each $b_1 \ge b_0$ any pair of approximations $\phi_1,\phi_2: \sd^{b_1 + b}(|K| \times |K|) \rightarrow |K|$ of the projections $\pi_1,\pi_2$ respectively are equivariantly $c$-contiguous for some $c \ge 0$. 
\end{proof}

%\begin{proof}
%Follows from Propositions \ref{prop:TCatmostSC} and \ref{prop:SCatmostTC}.
%\end{proof}

This result gives us a few useful corollaries as well.

\begin{corollary}%[Finite Simplicial Complexity]
Let $P$ be a finite $G$-connected poset and $K = \scr{K}(P)$ be the resulting ordered $G$-simplicial complex. Then 
$$
\SC_G(K) < \infty.
$$
\end{corollary}

\begin{proof}
The result follows from Theorem \ref{thm:SCeqTC}  and Proposition \ref{prop:finiteTCRealize}.
\end{proof}

\begin{corollary}
Suppose $P$ is a finite $G$-poset. Then 
$$
\SC_G(\scr K(P)) = \SC_G(\sd^b(\scr K(P)))
$$
for all $b \ge 0$. 
\end{corollary}

\begin{proof}
By Theorem~\ref{thm:SCeqTC} we have the first and third equalities below. The second equality holds because $|\scr K(P)| = |\sd^b(\scr K(P))|$ for all $b \ge 0$. Then
\[
 \SC_G(K) = \TC_G(|K|) = \TC_G(|\sd^b(K)|) = \SC_G(\sd^b(K)).
\]
\end{proof}

\printbibliography

\end{document}